\documentclass[12pt]{amsart}
\usepackage[active]{srcltx}
\usepackage{calc,amssymb,amsthm,amsmath,amscd, eucal,ulem,stmaryrd}
\usepackage{alltt}
\usepackage[left=1.35in,top=1.25in,right=1.35in,bottom=1.25in]{geometry}

\RequirePackage[dvipsnames,usenames]{color}

\input{kmacros3.sty}
\input{xy}
\xyoption{all}
\usepackage{tikz}
\usepackage{mabliautoref}

\DeclareSymbolFont{rsfs}{OMS}{rsfs}{m}{n}
\DeclareSymbolFontAlphabet{\scr}{rsfs}

\renewcommand{\fram}{\mathfrak{m}}
\numberwithin{equation}{theorem}

\usepackage{fullpage}

\usepackage{setspace}
\usepackage{hyperref}

\usepackage{enumerate}
\usepackage{graphicx}
\usepackage[all,cmtip]{xy}

\usepackage{upgreek}

\usepackage{verbatim}

\def\todo#1{\textcolor{Mahogany}%
{\footnotesize\newline{\color{Mahogany}\fbox{\parbox{\textwidth-15pt}{\textbf{todo: } #1}}}\newline}}

\begin{document}

\title{Global generation of test ideals in mixed characteristic and applications}
\author{Christopher Hacon, Alicia Lamarche, Karl Schwede}
\address{Department of Mathematics, University of Utah, Salt Lake City, UT 84112}
\begin{abstract}Suppose that $X$ is an integral scheme (quasi-)projective over a complete local ring of mixed characteristic.  Using ideas of Takamatsu-Yoshikawa and Bhatt-Ma-et. al, we define a notion of a $\bigplus$-test ideal on $X$, including for divisors  and linear series.  We obtain global generation results in this setting  that generalize the well known global generation results obtained via multiplier ideal sheaf techniques in characteristic $0$ and via test ideals in characteristic $p>0$.  We also obtain applications to the order of vanishing of linear series and to the diminished base locus in mixed characteristic similar to results of Ein-Lazarsfeld-\mustata-Nakamaye-Popa, Nakayama, and \mustata{} in the equal characteristic case.
\end{abstract}
\thanks{C.~Hacon was supported by NSF Grant
\#1801851, NSF FRG  Grant \#1952522 and by a grant from the Simons Foundation; Award Number:
256202.}
\thanks{A.~Lamarche was supported by NSF RTG Grant \#1840190 and an NSF Postdoctoral Fellowship.}
\thanks{K.~Schwede was supported by NSF Grants \#1801849 and \#2101800, NSF FRG  Grant \#1952522  and  a  Fellowship  from  the  Simons Foundation.}
\maketitle

\section{Introduction}

One of the most useful propeties of multiplier ideals $\mJ(X, \Delta)$ on a projective variety in characteristic zero is Nadel vanishing (a variant of Kawamata-Viehweg vanishing).  It states that $H^i(X, \cO_X(K_X + M) \otimes \mJ(X, \Delta)) = 0$ for $i > 0$ assuming that $K_X+M$ is a Cartier divisor such that $M -   \Delta$ is ample, \cite{NadelMultiplierIdealSheaves,LazarsfeldPositivity2}.  This in turn implies effective global generation results, for instance on a nonsingular projective variety $X$ of dimension $n$ we have that 
\[
    \mJ(X, \Delta) \otimes \cO_X(K_X + M + m L)
\]
is globally generated, for $m \geq n$ when $L$ is an ample divisor such that $\cO_X(L)$ is globally generated, see \cite[Proposition 9.4.26]{LazarsfeldPositivity2}.

Building on global generation results for $\cO_X(K_X + mL)$ for mildly singular varieties in positive characteristic (\cite{SmithFujitaFreenessForVeryAmple,KeelerFujita,HaraInterpretation}), similar results were also obtained for test ideals $\tau(X, \Delta)$ in characteristic $p > 0$ where instead of using resolution of singularities and Nadel/Kawamata-Viehweg vanishing, one uses Frobenius and Serre/Fujita vanishing.  Indeed, we have that if $M$ and $L$ are as above, then 
\[
    \tau(X, \Delta) \otimes \cO_X(K_X + M + m L)
\]
is not only globally generated, but is globally generated by the subspace of Frobenius stable sections \[ S^0(X, \Delta, \cO_X(K_X + M + mL)) \subseteq H^0(X, \cO_X(K_X + M + mL)),\] see \cite{SchwedeACanonicalLinearSystem,MustataNonNefLocusPositiveChar}.  This global generation result has found several important applications, including  towards understanding the non-nef locus (or diminished base locus) in positive characteristic in \cite{MustataNonNefLocusPositiveChar}.

In mixed characteristic, \ie for schemes defined over $\bZ_p$, we have theories of test/multiplier ideals \cite{MaSchwedePerfectoidTestideal} for complete local Noetherian rings.  However, their properties are still not well understood, for example, we still do not know whether they behave well under localization.  For projective schemes over a mixed characteristic complete local Noetherian ring, we also have the space of $\bigplus$-stable (in other words, trace invariant) sections
\[
    \myB^0(X, \Delta, \cO_X(K_X + M + mL))\subseteq H^0(X, \cO_X(K_X + M + mL)),
\] 
roughly analogous to $S^0$ above and which was successfully applied to the study of the 3-fold minimal model program in mixed characteristic \cite{TY-MMP,BMPSTWW-MMP}.  The usefulness of these special sections $\myB^0$ is derived from Bhatt's vanishing theorem \cite{BhattAbsoluteIntegralClosure,BMPSTWW-MMP} which can be viewed as a variant of Kawamata-Viehweg vanishing which holds on $X^+$ (the absolute integral closure of $X$). In this paper, we use these global sections $\myB^0$ to define a mixed characteristic test ideal.  

Suppose that $X \to \Spec R$ is a normal integral scheme which we assume is projective over a complete local Noetherian ring $(R, \fram)$.  For simplicity of statements in the introduction, we assume that $K_X$ is Cartier (for instance, if $X$ is Gorenstein).  We \emph{define} the ideal $\tau_+(\cO_X, \Delta) \subseteq \cO_X$ so that $\myB^0(X, \Delta, \cO_X(K_X + M))$ globally generates 
\[
    \tau_+(\cO_X, \Delta) \otimes \cO_X(K_X + M)
\]
for any sufficiently ample divisor $M$.  

We then show that this definition is independent of the choice of the sufficiently ample divisor $M$ in \autoref{prop.Tau+IndependentOfAmple}.   When $K_X + \Delta$ is $\bQ$-Cartier, we prove it agrees with the test ideal if $X$ is of equal characteristic $p > 0$ in \autoref{thm.AgreesWithCharPTestIdeal}.  We show it transforms under finite maps with the same formula as the test ideal and multiplier ideal in \autoref{thm.TransofmrationRulesUnderFiniteMaps}.  We also introduce an adjoint variant $\adj_+^D(\cO_X, \Delta)$ and show it restricts to $\tau_+(\cO_{D^N}, {\rm Diff}_{D^N}(D+\Delta))$ on the normalization $D^N$ of $D$ producing an adjunction/inversion of adjunction result in \autoref{thm.AdjunctionForAdjTau}.  

One of the obstructions to the theory of mixed characteristic test ideals as introduced in \cite{MaSchwedeSingularitiesMixedCharBCM} was the fact that the test ideals did not behave well under localization.  However, we are able to prove that the sheaf $\tau_+(\cO_X, \Delta)$ we study is defined affine locally on $X$, at least when $K_X + \Delta$ is $\bQ$-Cartier.  That is, we show the following:

\begin{theoremA*}[{\autoref{sec.Tau+OnAffineCharts}}]
    Suppose $R$ is a complete local Noetherian domain and $T = R[x_1, \dots, x_m]/J$ is a normal finite type domain over $R$.  Suppose that $U = \Spec T$ and $\Delta_U$ is a divisor on $U$ so that $K_U + \Delta_U$ is $\bQ$-Cartier.  Suppose then that $X_1, X_2$ are two integral normal projective compactifications of $U$ over $\Spec R$ with $\Delta_1, \Delta_2$ on $X_1, X_2$ such that $K_{X_i} + \Delta_i$ is $\bQ$-Cartier and $\Delta_i |_U = \Delta_U$.  Then we have that 
    \[
        \tau_+(\cO_X, \Delta_1)|_U = \tau_+(\cO_X, \Delta_2)|_U.
    \]        
\end{theoremA*}

It follows that we can define $\tau_+(U, \Delta_U)$ to be this restriction.  Furthermore, the formation of $\tau_+(U, \Delta_U)$ commutes with localization by a single element of $T$ (in other words, restricting to a smaller affine subset of $U$).  Finally, since any quasi-projective scheme  $Y$ over $\Spec R$ is the union of such affine schemes, we have a well-defined notion of $\tau_+(\cO _Y, \Delta)$. 

We prove the following optimal global generation property, hinted at on the first page:
\begin{theoremB*}[{{\autoref{cor.GlobalGenerationViaTauOX}}}]
    With notation as above, assuming that $ K_X+M$ is a Cartier divisor such that $M - \Delta$ is big and semi-ample, $\cO_X(L)$ is a globally generated and ample line bundle, and $m \geq n$ where $n = \dim X_{\fram}$ is the dimension of the closed fiber, then
    \[
        \tau_+(\cO_X, \Delta) \otimes \cO_X(K_X + M + mL)
    \]
    is globally generated by $\myB^0(X, \Delta, \cO_X(K_X + M + mL))$.  
\end{theoremB*}
\noindent
%
Note that our method of proof is via a Skoda-type complex, see \cite[Section 9.6.C]{LazarsfeldPositivity2}.  A similar method was also used in characteristic $p > 0$ in \cite{SchwedeTuckerTestIdealsNonPrincipal}.  In characteristic zero, a common proof of the analogous result is obtained via Castelnuovo-Mumford regularity \cite{LazarsfeldPositivity2}, and that approach was also used in characteristic $p > 0$ in \cite{KeelerFujita,SchwedeACanonicalLinearSystem, MustataNonNefLocusPositiveChar, SatoStabilityOfTestIdeals}.   Related results were also obtained in positive and mixed characteristic \cite{SmithFujitaFreenessForVeryAmple,BMPSTWW-MMP}  by using the fact that the (completion of the) graded absolute integral closure of the graded section ring is big Cohen-Macaulay.  

In fact, inspired by our Skoda-type argument, we also introduce a notion of test ideals of ideal sheaves $\tau(\cO_X, \fra^t)$ and show a Skoda-type result for these sheaves, see \autoref{thm.SkodaWithoutTooMuchGenerality}, and deduce a new Brian{\c{c}}on-Skoda type theorem in mixed characteristic.

Using our global generation result, we follow the ideas of \cite{MustataNonNefLocusPositiveChar} in characteristic $p > 0$, which itself uses the ideas of \cite{NakayamaZariskiDecompAndAbundance,EinLazMusNakPopAsymptoticInvariantsOfBaseLoci}.  We obtain the following two results analogous to some of the main results of those papers.

\begin{theoremC*}[{{\autoref{thm.OrdIndependence}}}]
    Let $Z$ be a closed integral subscheme of a regular integral scheme $X$ projective over our complete local Noetherian ring ${\rm Spec }(R)$ and $D$ a big $\mathbb Q$-divisor. Then ${\rm ord}_Z(||D||)$ is independent of the representative in the numerical class of $D$.
\end{theoremC*}

\begin{theoremD*}[{\autoref{thm.ZInDiminishedBaseLocus}}]
      With notation as above, let $Z$ be a closed irreducible reduced subscheme of a regular scheme $X$. Assume that $R/\frak m$ is infinite.
    If $D$ is big then the following are equivalent
    \begin{enumerate}
        \item $Z\not \subset {\mathbf B}_-(D)$,
        \item there exists a divisor $G$ such that $Z\not \subset {\rm Bs}|mD+G|$ for all $m\geq 1$,
        \item there exists an integer $M>0$ such that ${\rm ord}_Z(|mD|)\leq M$ for all $m\gg 0$,
        \item  ${\rm ord}_Z(||D||)=0$, and
        \item $\tau _+(\omega _X,\lambda \cdot ||D||)$ does not vanish along $Z$ for all $\lambda > 0$.
    \end{enumerate}
\end{theoremD*}

Also see \cite{CacciolaDiBiagioAsymptoticBaseLociOnSingular,SatoStabilityOfTestIdeals,MurayamaTheGammaConstruction} from some related results on singular varieties in equal characteristic and compare with \cite[Conjecture 2.7]{BoucksomBroustetPacienza}.

Using the same method as the proof of the above result, we also observe that if $z \in X$ is a point (possibly non-closed) so that $X_z$ is regular and $\mult_z(\Delta) < 1$, then $\tau_+(\cO_X, \Delta)$ agrees with $\cO_X$ in a neighborhood of $z$, see \autoref{cor.MultiplicityBoundOnThreshold}.  Note similar statements are known in positive characteristic, see \cite[Theorem 1.1]{KadyrsizovaKenkelPageSinghSmith-ExtramalSings} where it is also shown that such bounds are essentially optimal, also see \cite{SatoStabilityOfTestIdeals}.

\subsection*{Acknowledgements}

The authors thank Bhargav Bhatt, Rankeya Datta, Linquan Ma, Zsolt Patakfalvi, Shunsuke Takagi, Kevin Tucker, Joe Waldron, Jakub Witaszek, Shou Yoshikawa for valuable discussions.  Zsolt Patakfalvi in particular suggested we explore whether we can define $\tau_+(\cO_X, \Delta)$ for quasi-projective schemes over $R$.  Shunsuke Takagi suggested we could leverage our Skoda-type results to obtain a new Brian{\c{c}}on-Skoda type theorem.  Takumi Murayama pointed out a number of useful references.  We thank Seungsu Lee, Linquan Ma, Takumi Murayama, Zsolt Patakfalvi and Shunsuke Takagi for very useful comments on previous drafts of this paper.

\section{Preliminaries}

\subsection{Local and Matlis duality}
    We briefly review local and Matlis duality.

    Suppose $(R, \fram, k)$ is a complete local ring and $E = E_R(k)$ is the injective hull of the residue field.  Recall that the exact functor $\bullet^{\vee} = \Hom(\bullet, E)$ is faithful, takes Noetherian modules to Artinian modules (and Artinian modules to Noetherian ones),  and for any finitely generated $R$-module (or complex in $D^b_{fg}(R)$) we have a natural isomorphism of functors from the identity to $\bullet^{\vee\vee}$.

    We fix a normalized dualizing complex in the sense of \cite{HartshorneResidues}, that is $\omega_R^{\mydot}$ has its first nonzero cohomology in degree $-\dim R$.  For any finite type scheme $f : X \to \Spec R$ we fix $\omega_X^{\mydot} = f^! \omega_R^{\mydot}$.  

    \begin{lemma}
        Suppose that $X$ is projective over a complete local ring $(R, \fram, k = R/\fram)$.  Suppose further that $\sF$ is a coherent sheaf on $X$.  Let $\bullet^{\vee} = \Hom_R(\bullet, E)$ denote Matlis duality where $E = E_{k}$ is the injective hull of the residue field.  Then we have the following natural isomorphism in the derived category:
        \[
            \myR \Gamma(X, \myR \sHom_X(\sF, \omega_X^{\mydot}))^{\vee} \cong \myR \Gamma_{\fram}(\myR \Gamma(X, \sF)).
        \]
        In particular, if $\sF = \sL^{-1}$ is a line bundle and $X$ is Cohen-Macaulay of dimension $n$ (so that $\omega_X^{\mydot} = \omega_X[n]$), then 
        \[
            \big(H^{i}(X, \omega_X \otimes \sL)\big)^{\vee} = H^{n-i}_{\fram}(\myR \Gamma(X, \sL^{-1})).
        \]
    \end{lemma}
    \begin{proof}    
        This is a combination of Grothendieck Duality and local Duality.  Indeed, we have
        \[
            \myR \Gamma(X, \myR \sHom_X(\sF, \omega_X^{\mydot}))^{\vee} \cong \myR \Hom_R(\myR \Gamma(X, \sF), \omega_R^{\mydot})^{\vee} \cong \myR \Gamma_{\fram}(\myR \Gamma(X, \sF))
        \]
        where the first isomorphism is Grothendieck duality and the second is local duality.    
    \end{proof}

    In fact, due to a spectral sequence in low degree terms argument, for $i = 0$ in the above, we do not even need to assume that $X$ is Cohen-Macaulay {\cite[2.2]{BMPSTWW-MMP}}. 

\subsection{$X^+$ and vanishing in mixed characteristic} 
\label{subsec.VanishingInMixedChar}

Suppose $X$ is an excellent integral scheme and fix $\overline{K(X)}$ to be the absolute integral closure of the fraction field of $X$.  By $X^+$ we mean the limit over all finite surjective morphisms $f : Y \to X$ with $Y$ normal and integral with a fixed embedding $K(Y) \subseteq \overline{K(X)}$ which determines the map $f$, note $X^+$ is a scheme by \cite[\href{https://stacks.math.columbia.edu/tag/01YV}{Tag 01YV}]{stacks-project} as all the maps in the system are affine.    If $\nu : X^+ \to X$ is the canonical map, we can describe the structure sheaf of $X^+$ as
\[
    \nu_* \cO_{X^+} = \colim_{f : Y \to X} f_* \cO_Y.
\]
It is important to note that $X^+$ is typically not Noetherian even if $X$ is.  If $X$ is excellent and reduced with irreducible components $X_1, \dots, X_t$, then by $X^+$ we mean the disjoint union of the $X_i^+$.  

As mentioned in the introduction, Bhatt's vanishing theorem can be thought of as a variant of Kawamata-Viehweg vanishing on $X^+$. 
Because of its utility in our treatment of a mixed characteristic test ideals, we record it here for future reference. 

\begin{theorem}[{\cite[3.7]{BMPSTWW-MMP}}]
 \label{BhattVanishing}
Suppose that $(T,x)$ is an excellent local ring of residue characteristic $p>0$. Let $\pi: X\rightarrow \Spec(T)$ be a proper map with $X$ integral. Suppose that $L\in \pic(X)$ is a big and semi-ample line bundle. Then for all $b<0$ and all $i<\dim(X)$, we have that $H^i\left(\myR \Gamma_x (\myR\Gamma(X^+, L^b))\right)=0$.
\end{theorem}

\subsection{Base loci}
Let $D$ be a divisor on a normal scheme $X$ which is projective over a local ring $(R,\frak m)$.
We recall that ${\rm Bs|D|}$ is used to denote the base locus of the linear system $|D|$.
This is the subset of $X$ defined by the intersection of all elements $G\in |D|$ (if $|D|=\emptyset$ then we let ${\rm Bs|D|}=X$).
 The stable base locus of a $\Q$-divisor $D$ is the subset of $X$ defined by ${\bf B}(D) =\bigcap_{m} {\rm Bs}|mD|$ where we consider $m \geq 1$ such that $mD$ has integer coefficients. Note that ${\bf B}(D)= {\rm Bs}|mD|$ for any $m>0$ sufficiently divisible (if $\kappa (D)=0$, then we set ${\bf B}(D)=X$).
 If $D$ is pseudo-effective, then ${\mathbf B}_-(D)$  denotes the diminished base locus
 \[
     \bigcup_A {\bf B}(D+A)
 \]
 where the union is over ample $\bQ$-divisors $A$. Note that ${\mathbf B}_-(D)$ is a countable union of closed subsets of $X$.
\section{$\bigplus$-stable global sections}

    We introduce a slight variant of the global section theory introduced in \cite{TY-MMP} and \cite{BMPSTWW-MMP} which focuses more on adjoint line bundles $\omega_X \otimes \sL$.  We use the following setting.

    \begin{notation}
        \label{notation.SettingUpX+}
        Suppose $(R, \fram)$ is a complete Noetherian local ring with residual characteristic $p > 0$ and suppose that $X \to \Spec R$ is a proper map from a normal integral scheme of dimension $d$. 
        Fix an algebraic closure of the fraction field of $X$, $\overline{K(X)}$ (that is, a geometric generic point of $X$) and we consider systems of finite maps $f : Y \to X$ as in \autoref{subsec.VanishingInMixedChar}.  
        Using the notation of \autoref{subsec.VanishingInMixedChar}, we fix $\nu : X^+ \to X$ the canonical map.

        For each $\bQ$-divisor $B$ on $X$, we set 
        \[
            \nu_* \cO_{X^+}(\nu^* B) = \colim_{f : Y \to X} f_* \cO_Y(\lfloor f^* B\rfloor )
        \]
        where $f : Y \to X$ runs over finite dominant maps, and hence obtain the sheaf $\cO_{X^+}(\nu^* B)$ on $X^+$.  Notice that for large enough $f : Y \to X$, we have that $f^* B$ has integer coefficients.
    \end{notation}

    \begin{definition}[$\bigplus$-stable sections]
        \label{def.+StableSections}
        Suppose $(R, \fram)$ is a complete Noetherian local ring with residual characteristic $p > 0$ and suppose that $X \to \Spec R$ is a proper map from a normal integral scheme of dimension $d$.  
        For a divisorial sheaf $\sM=\mathcal O_X(M)$ and $\bQ$-divisor $B \geq 0$, we define 
        \[
            \myB^0(X, B, \cO_X(K_X + M))
        \]
        to be the $R$-Matlis dual of the image
        \[
            \Image\Big( H^d_{\fram}(\myR \Gamma(X,\cO_X(-M))) \to H^d_{\fram}(\myR \Gamma(X^+, \cO_{X^+}(-\nu^* M + \nu^* B))) \Big).
        \]        
        In particular, it is a submodule of $H^0(X, \cO_X(K_X + M))$, the Matlis dual of $H^d_{\fram}(\myR \Gamma(X,\cO_X(-M)))$.        
    \end{definition}

    This notation varies slightly from what is presented in \cite{BMPSTWW-MMP}, where they work with $\myB^0(X, B, \cO_X(N))$  and hence consider the Matlis dual of the image of
    \[
        H^d_{\fram}(\myR \Gamma(X,\cO_X(-N+K_X))) \to H^d_{\fram}(\myR \Gamma(X^+, \cO_{X^+}(\nu^* (-N + K_X  + B)))).
    \]
  In this paper, we replace $N$ by $ K_X+M$ and so we obtain the formula in \autoref{def.+StableSections}.

    We have the following immediate properties.

    \begin{lemma}
        \label{lem.SimplePropertiesOfB0}
        With notation as in \autoref{def.+StableSections}:
        \begin{enumerate}
            \item If $B' \geq B$, then $\myB^0(X, B', \cO_X(K_X + M)) \subseteq \myB^0(X, B, \cO_X(K_X + M))$.\label{lem.SimplePropertiesOfB0.1}
            \item $\myB^0(X, B, \cO_X(K_X + M))$ only depends on the linear equivalence class of $M$.  That is, if $M \sim M'$, an isomorphism $H^0(X, \cO_X(M)) \cong H^0(X, \cO_X(M'))$ induces 
                \[
                    \myB^0(X, B, \cO_X(K_X + M)) \cong \myB^0(X, B, \cO_X(K_X + M')).
                \]%
                \label{lem.SimplePropertiesOfB0.2}%
            \item If $F \geq 0$ is a Weil divisor, then there is a commutative diagram
                \[
                    \xymatrix{
                         \myB^0(X, B, \cO_X(K_X + M - F)) \ar[d] \ar[r]^{\sim} & \myB^0(X, B + F, \cO_X(K_X + M)) \ar[d] \\
                         H^0(X, \cO_X(K_X + M - F)) \ar@{^{(}->}[r] & H^0(X, \cO_X(K_X + M))
                    }
                \]
                where the map in the top row  is an isomorphism. 
                \label{lem.SimplePropertiesOfB0.3}
        \end{enumerate}
    \end{lemma}
    \begin{proof}
        Claims \autoref{lem.SimplePropertiesOfB0.1} and \autoref{lem.SimplePropertiesOfB0.2} follow directly from the definition.  For claim \autoref{lem.SimplePropertiesOfB0.3}, notice that because the sheaf $\cO_X(-M+F) / \cO_X(-M )$ has support of dimension less than $d=\dim X$, we have that 
        \[
            H^d_{\fram}(\myR \Gamma(X, \cO_X(-M ))) \to H^d_{\fram}(\myR \Gamma(X, \cO_X(-M+F)))
        \]
        surjects by \cite[\href{https://stacks.math.columbia.edu/tag/0A4R}{Tag 0A4R}]{stacks-project}.  Hence they both have the same image in 
        \[
            H^d_{\fram}(X^+, \cO_{X^+}(-\nu^*( M-F) + \nu^* B )) = H^d_{\fram}(X^+, \cO_{X^+}(-\nu^* M  + \nu^*(B + F) )).
        \]
        This proves \autoref{lem.SimplePropertiesOfB0.3}.
    \end{proof}

    \begin{remark}[Making $K_X + M$ Cartier]
        \label{rem.MakingKX+MCartier}
        Frequently, it is desirable that the divisor $K_X+M$ considered above be Cartier. 
         If $K_X+M$ is a Weil divisor, we may reduce to the case where $K_X+M$ is Cartier at least when $X$ is \emph{projective} over $\Spec R$.  Simply select $D > 0$ a Weil divisor such that $M + K_X + D$ is Cartier and use \autoref{lem.SimplePropertiesOfB0} \autoref{lem.SimplePropertiesOfB0.3} above to see that $\myB^0(X, B+D, \cO_X(K_X + M+ D)) = \myB^0(X, B, \cO_X(K_X + M))$.
    \end{remark}

\subsection{The PLT / adjoint variant of $\myB^0$}
    We now move on to the PLT / adjoint ideal variant of $\myB^0$.  Suppose $D=\sum_{i = 1}^t D_i$ is a reduced divisor.  We shall also have $B \geq 0$ a $\bQ$-divisor, where frequently $B$ and $D$ have no common components.  
    For each prime divisor $D_i$, with associated prime ideal sheaf $\cO_{X}(-D_i)$, we fix an integral subscheme $D_i^+ \subseteq X^+$ lying over $D_i$.  This gives us a prime ideal sheaf lying over $\cO_X(-D_i)$ which we write as $\cO_{X^+}(-D_i^+)$.  It is also a colimit of compatible ideal sheaves $\cO_Y(-D_{i, Y})$ on each finite dominant $f : Y \to X$ factoring $X^+ \to X$.

    \begin{definition}
        \label{def.AdjointStableSections}
        With notation as above, we set $\myB^0_D(X, D+B, \cO_X(K_X + M))$ to be the $R$-Matlis dual of the image        
        \[
            \Image\Big( H^d_{\fram}(\myR \Gamma(X, \cO_X(-M ))) \to \bigoplus_{i = 1}^t H^d_{\fram}(X^+, \cO_{X^+}(-\nu^* M + \nu^*(D + B) - D_i^+ ))\Big).    
        \]            
        It is an $R$-submodule of $H^0(X, \cO_X(K_X + M))$.      
    \end{definition}
    \begin{remark}[A variant of the definition when the boundary has no $D$]
        \label{rem.DefinitionOfAdjGeneralVersion}
        Notice that the sheaf $\cO_X(-M )/ \cO_X(-M-D)$ has support of dimension $< d$ and so 
        \[
            H^d_{\fram}(\myR \Gamma(X, \cO_X(-M-D)))  \twoheadrightarrow H^d_{\fram}(\myR \Gamma(X, \cO_X(-M ))) 
        \]
        is surjective.  In particular, 
        \[
            \Image\Big( H^d_{\fram}(\myR \Gamma(X, \cO_X(-M -D ))) \to \bigoplus_{i = 1}^t H^d_{\fram}(X^+, \cO_{X^+}(-\nu^* M + \nu^*(D + B) - D_i^+ ))\Big)
        \]
        coincides with the image in \autoref{def.AdjointStableSections}.  Its Matlis dual is then the subset of $H^0(X, \cO_X(K_X + M + D))$, given by the image of  $\myB^0_D(X, D+B, \cO_X(K_X + M))$ via the inclusion $H^0(X, \cO_X(K_X + M)) \subseteq H^0(X, \cO_X(K_X + M + D))$.
        
        We could consider the $R$-Matlis dual of:
        \begin{equation}
            \label{eq.AlternateB^0_DDefinitionImage}
            \Image\Big( H^d_{\fram}(\myR \Gamma(X, \cO_X(-M - D))) \to \bigoplus_{i = 1}^t H^d_{\fram}(X^+, \cO_{X^+}(-\nu^* M + \nu^*B - D_i^+ ))\Big).
        \end{equation}
        That is a subset of $H^0(X, \cO_X(K_X + D + M))$.  
        On the other hand, the image of \autoref{eq.AlternateB^0_DDefinitionImage} is the same as the image of 
        \[
            \Image\Big( H^d_{\fram}(\myR \Gamma(X, \cO_X(-M - D))) \to \bigoplus_{i = 1}^t H^d_{\fram}(X^+, \cO_{X^+}(-\nu^* (M + D) + \nu^*(B + D) - D_i^+ ))\Big)
        \]
        whose Matlis dual is by definition $\myB^0_D(X, D+B, \cO_X(K_X + D + M))$.
    \end{remark}
    
    In the case that $D$ is not prime, we must pay special attention to the direct sums considered above.  Note that the use of the direct sum separates the contribution from each $D_i$ (this partially explains why this a PLT variant).
    Additionally, because the Galois action is independent on each direct summand, we see that $\myB^0_D$ is independent of the choice of each $D_i^+$, see \cite[\href{https://stacks.math.columbia.edu/tag/0BRK}{Tag 0BRK}]{stacks-project}.

    \begin{lemma}
        \label{lem.SimplePropertiesOfB0_D}
        With notation as in \autoref{def.AdjointStableSections}:
        \begin{enumerate}
            \item If $B' \geq B$, then $\myB^0_D(X, D+B', \cO_X(K_X + M)) \subseteq \myB^0_D(X, D+B, \cO_X(K_X + M))$.\label{lem.SimplePropertiesOfB0_D.1}
            \item $\myB^0_D(X, B+D, \cO_X(K_X + M))$ only depends on the linear equivalence class of $M$.  That is, if $M \sim M'$, an isomorphism $H^0(X, \cO_X(K_X+M)) \cong H^0(X, \cO_X(K_X+M'))$ induces 
            \[
                \myB^0_D(X, B+D, \cO_X(K_X + M)) \cong \myB^0_D(X, B+D, \cO_X(K_X + M')).
            \]
            \label{lem.SimplePropertiesOfB0_D.2}
            \item For any $0 < \epsilon \leq 1$ we have that 
                \[
                    \myB^0(X, D+B, \cO_X(K_X + M)) \subseteq \myB^0_D(X, D+B, \cO_X(K_X + M)) \subseteq \myB^0(X, (1-\epsilon)D+B, \cO_X(K_X + M)).
                \]
                \label{lem.SimplePropertiesOfB0_D.3}
            \item If $F \geq 0$ is a Weil divisor, then there is a commutative diagram
                \[
                    \xymatrix{
                         \myB^0_D(X, B+D, \cO_X(K_X + M - F)) \ar[d] \ar[r]^{\sim} & \myB^0_D(X, B + D + F, \cO_X(K_X + M)) \ar[d] \\
                         H^0(X, \cO_X(K_X + M - F)) \ar@{^{(}->}[r] & H^0(X, \cO_X(K_X + M))
                    }
                \]
                 inducing an isomorphism of the top row. 
                \label{lem.SimplePropertiesOfB0_D.4}
        \end{enumerate}
    \end{lemma}
    \begin{proof}
        As in \autoref{lem.SimplePropertiesOfB0} we see that \autoref{lem.SimplePropertiesOfB0_D.1} and \autoref{lem.SimplePropertiesOfB0_D.2} follow  from the definition.  

        Claim \autoref{lem.SimplePropertiesOfB0_D.3} follows from the inclusions:
        \[
            \cO_{X^+}(-\nu^* M +  \nu^*((1-\epsilon)D+B) ) \subseteq \cO_{X^+}(-\nu^* M + \nu^*(D+B) - D_i^+) \subseteq \cO_{X^+}(-\nu^* M + \nu^*(D+B) ).
        \]
        
        Part \autoref{lem.SimplePropertiesOfB0_D.4} follows in the same way as \autoref{lem.SimplePropertiesOfB0} \autoref{lem.SimplePropertiesOfB0.3}.        
    \end{proof}

    We recall the following dual formulations of $\myB^0$ and $\myB^0_D$.

    \begin{lemma}[{{\cite[4.8, 4.13, 4.24, 4.25]{BMPSTWW-MMP}}}]
        \label{lem.BasicPropertiesOfB^0AndB^0_D}
        With notation as in \autoref{def.+StableSections}, and where $f : Y  \to X$ are as in \autoref{notation.SettingUpX+}, the following are equal:
        \begin{enumerate}
            \item $\myB^0(X, B, \cO_X(K_X + M))$.
            \item 
            \[ 
                \bigcap_{f:Y\to X}{\rm im}\left( H^0(Y,\mathcal \cO_Y(K_Y+\lceil f^*(M-B)\rceil ))\to H^0(X,\cO_X(K_X + M)) \right).
            \]                  
            \item The projection to $H^0(X, \cO_X(K_X + M))$ from the inverse limit as below:
            \[
                \Image \Big(\lim_{f : Y \to X} H^0(Y,\mathcal \cO_Y(K_Y+\lceil f^*(M-B)\rceil ) ) \to H^0(X,\cO_X(K_X + M)) \Big).
            \]
        \end{enumerate}
        Finally, in the case that $M - B$ is $\bQ$-Cartier, we may instead let $f : Y \to X$ run over alterations from integral normal schemes (with a chosen embedding $K(Y) \subseteq \overline{K(X)}$ as in the finite case.  
        
        The analogous result for the adjoint variants $\myB^0_D$ also holds and the following are equal:
        \begin{enumerate}
        \item $\myB^0_D(X, B+D, \cO_X(K_X + M))$.
            \item 
            \[ 
                \bigcap_{f:Y\to X}{\rm im}\left( H^0\Big(Y,\bigoplus_{i=1}^t \cO_Y(K_Y+D_{i,Y}+\lceil f^*(M-B-D)\rceil )\Big)\to H^0(X,\cO_X(K_X + M)) \right).
            \]                  
            \item The projection to $H^0(X, \cO_X(K_X + M))$ from the inverse limit as below:
            \[
                \Image \Big(\lim_{f : Y \to X} H^0(Y, \bigoplus_{i=1}^t \cO_Y(K_Y+D_{i,Y}+\lceil f^*(M-B-D)\rceil ) ) \to H^0(X,\cO_X(K_X + M)) \Big).
            \]

        \end{enumerate}         
    \end{lemma}
    \begin{proposition}
        \label{prop.B0StableBirational}
        Suppose that $\pi : X' \to X$ is a proper birational map between normal schemes with $X, B, M$ as above where $B$ is $\bQ$-Cartier and $M$ is Cartier.  Then we have that 
        \[
            \myB^0(X, B, \cO_X(K_X + M)) \cong \myB^0(X', \pi^* B, \cO_{X'}(K_{X'} + \pi^* M)).
        \]
    \end{proposition}
    \begin{proof}
A similar argument can be found in the proof of \cite[Lemma 4.19]{BMPSTWW-MMP}. Recall that a proper birational map between normal schemes is an alteration. For any alteration $f:Y\rightarrow X'$ we have the following commutative diagram.  \[
                    \xymatrix{
                         H^0\left(Y, \cO_Y(K_Y + \lceil f^\ast(\pi^\ast(M-B))\rceil)\right)  \ar[d] \ar[r] & H^0\left(X', \cO_{X'}(K_{X'}+\pi^\ast M)\right) \ar[d]\\
                         H^0\left(Y, \cO_Y(K_Y + \lceil (\pi \circ f)^\ast(M-B)\rceil)\right) \ar[r] & H^0(X, \cO_X(K_X + M))
                    }
                \]
        Applying \autoref{lem.BasicPropertiesOfB^0AndB^0_D} to this diagram yields the desired claim. 
    \end{proof}

    We will need the following result.
    \begin{proposition}
        \label{prop.FiniteEtalExtensionOfRAndB^0}
        With notation as above, suppose that $H^0(X, \cO_X) = R \subseteq S$ is finite \'etale where $S$ is also a domain (and hence local and complete).  Suppose that $M$ is Weil a divisor on $X$.  Let $X_S = X \times_R S$ denote the base change, likewise with $B_S$ and $M_S$.  Then, 
        \[
            \myB^0(X, B, \cO_X(K_X + M)) \otimes_R S = \myB^0(X_S, B_S, \cO_{X_S}(K_{X_S} + M_S)).  
        \]
    \end{proposition}    
    \begin{proof}
        Consider the map 
        \begin{equation}
            \label{eq.ImageOfMapBeforeTensor}
            H^d_{\fram}(\myR \Gamma(X,\cO_X(-M))) \to H^d_{\fram}(\myR \Gamma(X^+, \cO_{X^+}(-\nu^* M + \nu^* B)))
        \end{equation}
        whose image $I_1$ is dual to $\myB^0(X, B, \cO_X(K_X + M))$.  Tensoring this map with $S$ yields:
        \[
            H^d_{\fram}(\myR \Gamma(X,\cO_X(-M)) \otimes_R S) \to H^d_{\fram}(\myR \Gamma(X^+, \cO_{X^+}(-\nu^* M + \nu^* B))\otimes_R S).
        \]
        The left side is $H^d_{\fram}(\myR \Gamma(X_S,\cO_{X_S}(-M_S)))$ but the right side is slightly more complicated.  The point is, choosing an embedding of $K(X_S) \subseteq \overline{K(X)}$, we see that $X^+ \to X$ already factors through $X_S$ and so $(X^+)_S = X^+ \times_{\Spec R} \Spec S$ will break up into multiple components since even $S \otimes_R S$ will not be a domain.  
        
        Indeed, let $S' \supseteq S \supseteq R$ denote the integral closure of $R$ in the Galois closure $K(S')$ of $K(S)$ over $K(R)$ (in $K(X^+)$). Then $(X^+)_S = X^+ \times_{\Spec S'} \Spec S' \times_{\Spec R} \Spec S$.  Notice that in fact $S' \otimes_R S$ is a direct sum of $[K(S) : K(R)]$ copies of $S'$.  Hence, we see that $(X^+)_S$ is simply a product of $X^+$'s.  
        
        Furthermore, the map $X_S \to (X^+)_S$ will be the product of maps corresponding to the different ways to embed $K(X_S) \subseteq K(X^+) = \overline{K(X)}$.  Since any such embeddings will differ by a Galois action on $K(X^+)$, we see that each such induced homomorphism
        $H^d_{\fram}(\myR \Gamma(X_S,\cO_{X_S}(-M_S))) \to H^d_{\fram}(\myR \Gamma(X^+, \cO_{X^+}(-\nu^* M + \nu^* B)))$
        has the same kernel (and hence  isomorphic images).  Thus we see that the $S$-Matlis dual of the image of 
        \begin{equation}
            \label{eq.ImageOfMapAfterTensor}
            H^d_{\fram}(\myR \Gamma(X,\cO_X(-M)))\otimes_R S \to H^d_{\fram}(\myR \Gamma(X^+, \cO_{X^+}(-\nu^* M + \nu^* B)))\otimes_R S.
        \end{equation}
        is indeed $\myB^0(X_S, B_S, \cO_{X_S}(K_{X_S} + M_S))$.  The image $I_2$ of \autoref{eq.ImageOfMapAfterTensor} is also the tensor product of the image  $I_1$ of \autoref{eq.ImageOfMapBeforeTensor} since $R \to S$ is flat.  In other words $I_2 = I_1 \otimes_R S$.  

        Let $n = \dim R = \dim S$.  Notice that $E_S = \Hom_R(S, E_R)$ since $R \to S$ is finite and so:
        \[
            \Hom_S(I_2, E_S) \cong \Hom_R(I_2, E_R) = \Hom_R(S \otimes_R I_1, E_R) = S \otimes \Hom_R(I_1, E_R)
        \]
        Here the first isomorphism is true since $R \to S$ is finite and the final isomorphism can be found in \cite[Exercise 14 on Page 45]{Bourbaki1998} since $S$ is finite and $E_R$ is injective.
    \end{proof}

    We also recall \cite[Theorem 7.2]{BMPSTWW-MMP}.
    \begin{theorem}[Lifting sections]            
        \label{thm.LiftingSections}
        With notation as in \autoref{def.+StableSections}, suppose $ K_X+M$ is Cartier and that $M  - D - B$ is $\bQ$-Cartier, big and semi-ample.  Suppose that $D^{\Norm}\to D$ is the normalization and $\kappa : D^{\Norm} \to X$ is the canonical map and write $\kappa^* \cO_X(K_X + M) \cong \cO_{D^{\Norm}}(K_{D^{\Norm}} + M_{D^{\Norm}})$ for some Weil divisor $M_{D^{\Norm}}$.   Then the map
        $H^0(X, \cO_X(K_X + M)) \to H^0(D^{\Norm}, \cO_{D^{\Norm}}(K_{D^{\Norm}} + M_{D^{\Norm}}))$
         induces a surjection        
        \[
            \myB^0_D(X, D+B, \cO_X(K_X + M)) \twoheadrightarrow \myB^0(D^{\Norm}, \Diff_{D^{\Norm}}(B + D), \cO_{D^{\Norm}}(K_{D^{\Norm}} + M_{D^{\Norm}})).
        \]    
    \end{theorem}

    \begin{remark}[Lifting sections without assuming that $K_X + M$ is Cartier]
        Suppose that $X \to \Spec R$ is projective and that $K_X+M$ is not Cartier but that $ K_X+M + F$ is Cartier  for some effective divisor $F$ with no common components with $D$.  Replacing $M$ by  $M+F$  and $B$ by  $B+F$, we may apply \autoref{thm.LiftingSections}, to obtain a surjection 
        \[
            \myB^0_D(X, D+B + F, \cO_X(K_X + M +F)) \to \myB^0(D^{\Norm}, \Diff_{D^{\Norm}}(B + D + F), \kappa^* \cO_X(K_X + M +F)).
        \]

        But by \autoref{lem.SimplePropertiesOfB0_D} \autoref{lem.SimplePropertiesOfB0_D.4}, we see that
        the left side is isomorphic to $\myB^0_D(X, D+B, \cO_X(K_X + M))$.
        Thus $\myB^0_D(X, D+B, \cO_X(K_X + M))$ surjects onto a meaningful analog of $\myB^0(D^{\Norm}, \Diff_{D^{\Norm}}(B + D), \cO_{D^{\Norm}}(K_{D^{\Norm}} + M_{D^{\Norm}}))$. Note that the term $F$ contributes to the different (or perhaps more accurately, the fact that $K_X + M$ is not Cartier contributes to the different).
    \end{remark}

\section{A global mixed characteristic theory of test ideals}

    Based on the equal characteristic setting, at least for big enough twists by an ample line bundle, one may expect that the sections $\myB^0$ ought to globally generate certain important subsheaves of $\cO_X(K_X + M)$.  In equal characteristic, under suitable assumptions, analogs of $\myB^0$  generate multiplier ideals and test ideals tensored with $\cO_X(K_X + M)$ (see \cite[Chapter 10]{LazarsfeldPositivity2}, \cite{SchwedeACanonicalLinearSystem}, \cite{MustataNonNefLocusPositiveChar}).  
    
    In this section, we introduce the sheaves generated by $\myB^0$.  Note that the test ideals and modules introduced in \cite{MaSchwedeSingularitiesMixedCharBCM} require the ring we are working with to be complete.  While we believe some of these ideas still work for non-complete bases as in \cite{TY-MMP}, as of now we seem to get stronger results for complete bases, and so we restrict to that setting.

    Throughout this section, we work in the following setting.
    \begin{setting}\label{setting4.1}
        Fix $R$  a complete local ring with residual characteristic $p > 0$ and suppose that $X$ is a normal integral scheme projective over $R$.    We fix $\sL=\mathcal O _X(L)$ a very ample line bundle on $X$.
    \end{setting}

    Typically we also assume that $H^0(X, \cO_X) = R$.  Indeed, we may reduce to this case by replacing $R$ by $H^0(X, \cO_X)$ if necessary.  Since $H^0(X, \cO_X)$ is a finite extension of $R$, this does not impact the computation of Matlis duality.
    
    We aim to introduce a notion of test ideals on $X$.

\subsection{Test ideals via global sections}
    \label{subsec.AlternateConstruction}

    We begin with the following lemma.


      \begin{lemma}   
        \label{lemma.EasyInclusionForB0ByMultiplication}
        With notation as above, let $M$ and $N$ be divisors on $X$.  Then, the image of the canonical multiplication map:
        \[
            \myB^0(X, B, \cO_X(K_X + M)) \otimes_R H^0(X, \cO_X(N)) \to H^0(X, \cO_X(K_X + M + N))
        \]        
        is contained in $\myB^0(X, B, \cO_X(K_X + M + N))$. Similarly for $\myB^0_D$, the image of
        \[
            \myB^0_D(X, D+B, \cO_X(K_X + M)) \otimes_R H^0(X, \cO_X(N)) \to H^0(X, \cO_X(K_X + M + N))
        \]      
        is contained in $\myB^0_D(X, D+B, \cO_X(K_X + M + N))$.
    \end{lemma} 
    \begin{proof}
        Fix a choice of $R$-module generators of $y_1, \dots, y_m \in H^0(X, \cO_X(N))$.  We will pull these back to $\pi^* y_i$ global sections of $\cO_Y(\pi^* N)$ for each $\pi : Y \to X$ a finite cover from a normal integral scheme.  Now, suppose that  $z \in \myB^0(X, B, \cO_X(K_X + M))$, then for each finite morphism $\pi : Y \to X$ as before with $\pi^* B$ a Weil divisor,  there exists $z_Y \in H^0(Y, \cO_Y(K_Y - \pi^* B + \pi^* M))$ mapping to $z$ via the corresponding trace map.  Hence $z_Y \otimes \pi^* y_i \in H^0(Y, \cO_Y(K_Y - \pi^* B + \pi^* M + \pi^* N))$ maps to $z \otimes y_i$.  The argument for $\myB^0_D$ is the same.  This proves the lemma.  
    \end{proof}
    
    \begin{definition}[Test modules]\label{testmoduledefn}
        Suppose that $X$ and $\sL$ are as in \autoref{setting4.1}. For each $i$, we define the subsheaf $\sN_i \subseteq \omega_X\otimes \sL^i$ generated by $\myB^0(X, B,\omega_X\otimes\sL^i)\subseteq H^0(X, \omega_X\otimes\sL^i)$, and define $\sJ_i := \sN_i\otimes\sL^{-i}$ a subsheaf of $\omega_X$.  Because $\sL$ is globally generated, \autoref{lemma.EasyInclusionForB0ByMultiplication} implies that $\sJ_i\subseteq \sJ_{i+1}\subseteq \omega_X$, and hence $\sJ_i = \sJ_{i+1}$ for large enough $i$. One may then define the \emph{$\bigplus$-test submodule} $\tau_+^\sL(\omega_X, B)$ to be $\sJ_i$ for $i \gg 0$.  If $B = 0$, we simply write $\tau_+^\sL(\omega_X)$.
    \end{definition}
    \begin{definition}[Adjoint modules]\label{adjointmoduledefn}
        Likewise, with notation as in \autoref{def.AdjointStableSections}, the \emph{adjoint $\bigplus$-test module} $\adj_+^{D,\sL}(\omega_X(D),B)$ is defined as follows.  For each $i > 0$ let $\sN_i \subseteq \omega_X(D) \otimes \sL^i = \cO_X(K_X + D) \otimes \sL^i$ be the submodule generated by $\myB^0_D(X,B+D,\cO_X(K_X + D) \otimes \sL^i)$.  Setting, $\sI_{i} = \sN_i \otimes \sL^{-i} \subseteq \omega_X(D)$, we see that $\sI_i \subseteq \sI_{i+1}$ by \autoref{lemma.EasyInclusionForB0ByMultiplication}.  We define $\adj_+^{D,\sL}(\omega_X(D),B) := \sI_{i}$ for $i \gg 0$.
    \end{definition}

\subsection{Independence of choice of ample $\sL$}
    Our next goal is to show that $\tau_+^{\sL}(\omega_X, B )$ is independent of the choice of the ample line bundle $\sL$.  


    \begin{proposition}
        \label{prop.Tau+IndependentOfAmple}
        Given any two very ample line bundles $\sL$ and $\sL'$, we have that 
        \[
            \tau_+^{\sL}(\omega_X, B) = \tau_+^{\sL'}(\omega_X, B).
        \]
        Moreover, for any sufficiently ample line bundle $\sA$, then $\tau_+(\omega_X, B) \otimes \sA$ is generated by $\mathbf B ^0(X,B,  \omega _X\otimes \sA).$
        The analogous statements also hold for $\adj_+^{D, \sL}(\omega_X(D), B)$.
    \end{proposition}
    \begin{proof}
        It is clear from the definition that we have equality when $\sL' = \sL^k$ for any $k > 0$ since the $\sN_i$ are ascending.  
        Hence, taking high enough powers of $\sL'$ again if necessary, we may assume that $\sL' = \sL \otimes \sG$ where $\sG$ is a globally generated (even very ample) line bundle.

         \autoref{lemma.EasyInclusionForB0ByMultiplication}  implies that $\tau_+^\sL(\omega_X, B) \subseteq \tau_+^{\sL'}(\omega_X, B)$.  Reversing the roles of $\sL$ and $\sL'$ proves the reverse inclusion and hence we have shown that: 
         \[
            \tau_+^{\sL}(\omega_X, B) = \tau_+^{\sL'}(\omega_X, B).
        \]
        The proof for $\adj_+^D(\omega_X(D), B)$ is the same.
        
        Suppose now that $\sA$ is sufficiently ample and in particular $\sA\otimes \sL ^{-i}$ is very ample for some $i \gg 0$. Let $M \otimes \sA \subseteq \omega_X \otimes \sA$ be the module generated by
        $\mathbf B ^0(X,B,  \omega _X\otimes \sA)$. Arguing as above $\mathbf B ^0(X,B,  \omega _X\otimes \sL^i)\otimes H^0(X,\sA\otimes \sL ^{-i})$ is contained in $\mathbf B ^0(X,B,  \omega _X\otimes \sA)$ and hence $\tau^\sL  _+(\omega _X,B)\subseteq M$. On the other hand since $\mathbf B ^0(X,B,  \omega _X\otimes \sA)\otimes H^0(X,\sA ^{j-1})$ is contained in $\mathbf B ^0(X,B,  \omega _X\otimes \sA^j)$ for any $j>0$, it follows that 
        $M\subseteq \tau ^\sA_+(\omega _X,B)$. Since we have already seen that $\tau ^\sA_+(\omega _X,B)= \tau^\sL  _+(\omega _X,B)$, the claim follows.  The proof for $\adj_+^D(\omega_X(D), B)$ is analogous.
    \end{proof}

    \begin{notation}
        In view of \autoref{prop.Tau+IndependentOfAmple}, we omit the $\sL$ from the notation of $\tau_+^\sL(\omega_X, B)$ (respectively, $\adj_+^{D, \sL}(\omega_X(D), B)$) and instead simply write $\tau_+(\omega_X, B)$ (respectively, write $\adj_+^{D}(\omega_X(D), B)$).
    \end{notation}

      We make the following observation.

    \begin{proposition}
        \label{prop.B0ForBigiIsGlobalSections}
        For $i \gg 0$, we have that 
        \[
            H^0(X, \tau_+(\omega_X, B) \otimes \sL^i) = \myB^0(X, B, \omega_X \otimes \sL^i)
        \]
        and similarly that 
    
        \[
            H^0(X, \adj^{D}_+(\omega_X(D), B) \otimes \sL^i) = \myB^0_D(X, D+B, \omega_X(D) \otimes \sL^i).
        \]
    \end{proposition}
    \begin{proof}
        By \autoref{testmoduledefn} we have that $\tau_+(\omega_X, B)\otimes \sL^i = \sJ_i \otimes \sL^i = \sN_i$ is generated by $\myB^0(X, B, \omega_X\otimes \sL^i)$ for $i\gg0$.  We must then prove that  the induced inclusion $\myB^0(X, B, \omega_X\otimes \sL^i)\subseteq H^0(X, \tau_+(\omega_X, B) \otimes \sL^i)$ is an equality for $i\gg 0$. 

        Let $S = \bigoplus_{i\geq 0} H^0(X, \sL^i)$ be the graded ring whose degree zero part is the complete local ring $(R,\mathfrak{m})$. Fix $N\gg 0$. We consider the graded $S$-module $J:= \bigoplus_{i\geq N} \myB^0(X,B,\omega_X\otimes \sL^i)$. \autoref{lemma.EasyInclusionForB0ByMultiplication} proves that $J$ is a graded $S$-submodule of the graded canonical module $\omega_S$ (which in large degrees is $H^0(X, \omega_X \otimes \sL^i)$, see \cite[Section 5]{BMPSTWW-MMP}). The sheaf $\sJ$ associated to the graded module $J$ is a coherent subsheaf of $\omega_X$ on $X$.  Furthermore, we have that $H^0(X, \sJ \otimes \sL^i)$ is the $i$th graded piece $[J]_i$ of $J$ in sufficiently high degree $i$ by \cite[Theorem 2.3.1]{EGAIII1} (a similar statement is found in \cite[Exercise 2.5.9 (b)]{Hartshorne} under unnecessary geometric hypotheses).  Furthermore, $H^0(X, \sJ \otimes \sL^i)$ globally generates $\sJ \otimes \sL^i$ for $i \gg 0$.  Therefore $\sJ \otimes \sL^i = \tau_+(\omega_X, B)\otimes \sL^i$, since the right side is generated by $ \myB^0(X,B,\omega_X\otimes \sL^i)$.  This proves the desired result.  The adjoint statement in the corollary follows from the same reasoning with minor modifications.
    \end{proof}

    \begin{lemma}[Basic facts about $\tau_+$ and $\adj_+$]
        \label{lem.BasicFactsAboutTauAndAdj}
        With notation as above:
        \begin{enumerate}
            \item If $B' \geq B$ then $\tau_+(\omega_X, B') \subseteq \tau_+(\omega_X, B)$ and likewise $\adj_+^D(\omega_X(D), B') \subseteq \adj^D_+(\omega_X(D), B)$.
                \label{lem.BasicFactsAboutTauAndAdj.a}
            \item If $F\geq 0$ is a Cartier divisor, then $\tau _+(\omega_X, B+F) =\tau _+(\omega_X, B)\otimes \mathcal O _X(-F)$ and 
                $\adj_+^D(\omega_X(D), B+F)=\adj_+^D(\omega_X(D), B)\otimes \mathcal O _X(-F)$.
                \label{lem.BasicFactsAboutTauAndAdj.b}
            \item For any $0 < \epsilon \leq 1$ we have         
                \[
                    \tau_+(\omega_X, D+B) \subseteq \adj_+^D(\omega_X(D), D+B) \subseteq \tau_+(\omega_X, (1-\epsilon)D + B).
                \]            
                \label{lem.BasicFactsAboutTauAndAdj.c}
        \end{enumerate}
    \end{lemma}
    \begin{proof}
        These first part follow directly from the corresponding inclusions of $\myB^0$ and $\myB^0_D$ in \autoref{lem.SimplePropertiesOfB0} and \autoref{lem.SimplePropertiesOfB0_D}.  

        For \autoref{lem.BasicFactsAboutTauAndAdj.b}, fix $M$ sufficiently ample so that $M - F$ is also sufficiently ample.  Notice that $\myB^0(X, B + F, \cO_X(K_X + M))) \subseteq H^0(X, \cO_X(K_X + M))$ globally generates $\tau _+(\omega_X, B+F) \otimes \cO_X(M)$.  However, using \autoref{lem.SimplePropertiesOfB0} \autoref{lem.SimplePropertiesOfB0.3}, $\myB^0(X, B + F, \cO_X(K_X + M))) \cong \myB^0(X, B, \cO_X(K_X + M - F))) \subseteq H^0(X, \cO_X(K_X + M -F))$ also globally generates $\tau_+(\omega_X, B) \otimes \cO_X(M - F)$.  Hence $\tau_+(\omega_X, B) \otimes \cO_X(M - F) = \tau _+(\omega_X, B+F) \otimes \cO_X(M)$ which proves \autoref{lem.BasicFactsAboutTauAndAdj.b} for $\tau_+$.  The version for $\adj_+^D$ is the same except we use \autoref{lem.SimplePropertiesOfB0_D} \autoref{lem.SimplePropertiesOfB0_D.4}.  
        
        We now check \autoref{lem.BasicFactsAboutTauAndAdj.c}.  Observe first that the $R$-module $\myB^0_D(X, B+2D, \cO_X(K + D + M))$ globally generates $\adj^D_+(\omega_X(D), D+B) \otimes \cO_X(M)$ for $M$ sufficiently ample.

        Then we have by \autoref{lem.SimplePropertiesOfB0_D} \autoref{lem.SimplePropertiesOfB0_D.3} and \autoref{lem.SimplePropertiesOfB0_D.4} that 
        \[
            \begin{array}{rl}
                & \myB^0(X, D+B, \cO_X(K_X + M))\\
                = & \myB^0(X, 2D + B, \cO_X(K_X + D + M))\\
                \subseteq & \myB^0_D(X, 2D+B, \cO_X(K_X + D + M))\\
                \subseteq & \myB^0(X, D + (1-\epsilon)D+B, \cO_X(K_X + D + M))\\
                = & \myB^0(X, (1-\epsilon)D+B, \cO_X(K_X + M)).
            \end{array}
        \]
        The first term generates $\tau_+(\omega_X, D+B)$, the last generates $\tau_+(\omega_X, (1-\epsilon)D + B)$ and the third line generates $ \adj_+^D(\omega_X(D), D+B)$.  The result follows.    
    \end{proof}

Before we move on to comparing our test submodule with the multiplier submodule, we show that $\tau_+(\omega_X, B)$ is nonzero.  
\begin{lemma}
    \label{lem.TauIsNonzero}
    With notation as above, $\tau_+(\omega_X, B)$ and $\adj_+^D(\omega_X( D),  B)$ are nonzero.  
\end{lemma}
\begin{proof}
    We may assume that $X$ has mixed characteristic because if $X$ can be defined in characteristic $p > 0$, the result is well known.

    For the first statement, it suffices to show that $\myB^0(X, B, \omega_X \otimes \sL^i) \neq 0$ for $i \gg 0$.  Choose $D \geq B$ so that $D \sim m L$ for some $m > 0$.  It suffices to show that $0 \neq \myB^0(X, D, \omega_X \otimes \cO_X(iL))$ by \autoref{lem.SimplePropertiesOfB0} \autoref{lem.SimplePropertiesOfB0.1}.  Now, by \autoref{lem.SimplePropertiesOfB0} \autoref{lem.SimplePropertiesOfB0.2} and \autoref{lem.SimplePropertiesOfB0.3} we have that 
    \[
        \myB^0(X, 0, \cO_X(K_X + (i-m)L)) \cong \myB^0(X, 0, \cO_X(K_X + iL - D)) \cong \myB^0(X, D, \cO_X(K_X + iL))
    \]
    and so we must show that $\myB^0(X, 0, \cO_X(K_X + (i-m)L)) \neq 0$.  Thus, since $i \gg 0$, we must show that $\myB^0(X, 0, \cO_X(K_X + iL)) \neq 0$.  
    
    If $S = \bigoplus_{i \geq 0} H^0(X, \sL^i)$ is the section ring with respect to $\sL$, using \cite[Proposition 5.5]{BMPSTWW-MMP} we see that 
    \[
        \myB^0(X, 0, \cO_X(K_X + iL)) = [\tau_{S^{+,\gr}}(\omega_S)]_i
    \]
    where $\tau_{S^{+,\gr}}(\omega_S)$ is defined as in \cite[Section 5]{BMPSTWW-MMP}.  This is equivalent to showing that the map $H^{d+1}_{\fram_S}(S) \to H^{d+1}_{\fram_S}(S^{+,\gr})$ is nonzero.  
    It is thus sufficient to show that $\tau_{\widehat{S^{+,\gr}}}(\omega_{\widehat{S}}) \neq 0$ 
    since this is nonzero if and only if the map on local cohomology above is nonzero (this object is defined in \cite[Section 5]{MaSchwedeSingularitiesMixedCharBCM}).  
    Because we can take a Noether normalization $A \subseteq \widehat{R}$ via the Cohen-Structure theorem, we have that $\tau_{\widehat{S^{+,\gr}}}(\omega_{\widehat{S}}) \neq 0$ by \cite[Theorem 5.13]{MaSchwedeSingularitiesMixedCharBCM}.

    For the statement regarding $\adj_+^D(\omega_X(D), B)$,
    notice that 
    \[
        \begin{array}{rl}
            0 \neq & \myB^0(X, B+D, \omega_X \otimes \sL^i) \\
            \subseteq & \myB^0_D(X, B+D, \omega_X \otimes \sL^i)\\
             = & \myB^0_D(X, B+2D, \omega_X(D) \otimes \sL^i) \\
             \subseteq & \myB^0_D(X, B+D, \omega_X(D) \otimes \sL^i)
        \end{array}
    \]
    by \autoref{lem.SimplePropertiesOfB0_D}.  The latter generates $\adj_+^D(\omega_X(D), B)$.
\end{proof}

    \begin{remark}
        Notice that in the case that $B= 0$ and $X$ is nonsingular, we are not (yet) asserting that $\tau_+(\omega_X) = \omega_X$.  This is true, and follows from \autoref{prop.tauAgreesWithOmegaOnSNCLocus}, or from \cite[Theorem 5.8]{BMPSTWW-MMP}, but is not obvious based on our definitions.  
    \end{remark}

\subsection{Blowups and alterations}

    Our next goal is to discuss the behavior of  $\tau_+$ under proper birational maps, or even alterations and to relate $\tau_+$  with the multiplier ideal if a resolution of singularities exists.

    \begin{proposition}\label{p-tau}
        With notation as above, suppose additionally that $B$ is $\bQ$-Cartier.  For any alteration $\pi : Y \to X$, we have that 
        \[
            \tau_+(\omega_X, B) \subseteq \Tr \pi_* \cO_Y(K_Y -\lfloor \pi^* B \rfloor)
        \]
        where $\Tr : \pi_* \omega_Y \to \omega_X$ is the Grothendieck trace.
    \end{proposition}
    \begin{proof}
        By \autoref{prop.B0ForBigiIsGlobalSections}, it suffices to show that 
        \[
            \myB^0(X, B, \omega_X \otimes \sL^i) \subseteq H^0\big(X, (\Tr \pi_* \cO_Y(K_Y-\lfloor \pi^* B \rfloor)) \otimes \sL^i\big)
        \]
        for $i \gg 0$ and $\sL$ ample.  But for any $i$, 
        \[
            \begin{array}{rl}
                & \myB^0(X, B, \omega_X \otimes \sL^i)\\
                 \subseteq & \Image\Big(H^0(X, (\pi_* \cO_Y(K_Y - \lfloor \pi^* B \rfloor)) \otimes \sL^i ) \xrightarrow{\Tr}  H^0(X, \omega_X \otimes \sL^i) \Big)\\
                \subseteq & H^0\big(X, (\Tr \pi_* \cO_Y(K_Y-\lfloor \pi^* B \rfloor)) \otimes \sL^i\big).
            \end{array}
        \]
        which completes the proof.
    \end{proof}   

    Localizing to characteristic zero, we obtain the following.
    
    \begin{corollary}
        \label{cor.Tau+InvertPIsInMultiplier}
        Fix notation as in \autoref{p-tau}, including that $B$ is $\mathbb Q$-Cartier.  Let $R' = R[1/p]$, so for instance if $R$ is a mixed characteristic DVR, then $R'$ is a field of characteristic zero.
        Then 
        \[ 
            \tau_+(\omega_X, B)|_{X_{R'}}\subseteq \mathcal J(\omega_{X_{R'}} ,B |_{X_{R'}}) = \mathcal J(X_{R'} ,(-K_X +B) |_{X_{R'}})
        \]    
        where $X_{R'}$ is the base change.
    \end{corollary}
    \begin{proof} 
        Fix $\pi: Y \to X$ a proper birational map from a normal $Y$ such that $Y_{R'} \to X_{R'}$ is a log resolution of singularities.  Arguing as above, and using that $\Tr$ is simply an inclusion since $\pi$ is birational, we have 
        \[
            \tau_+(\omega_X, B)|_{X_{R'}}\subseteq 
            \Tr \pi_* \cO_Y(K_Y -\lfloor \pi^* B \rfloor)|_{X_{R'}} =
            (\pi|_{Y_{R'}})_* \cO_{Y_{R'}}(K_{Y_{R'}} -\lfloor \pi^* B|_{X_{R'} } \rfloor).
        \]
        The right side is $\mathcal{J}(\omega_{X_{R'}} ,B |_{X_{R'}})$ by definition.
    \end{proof}

    One expects that the containment in the statement of \autoref{cor.Tau+InvertPIsInMultiplier} is equality.  In fact, we have learned that forthcoming work by the authors of \cite{BMPSTWW-MMP} will address related questions.

    \autoref{p-tau} implies the following, also compare with \cite[Proposition 3.3]{MustataNonNefLocusPositiveChar}.

    \begin{lemma} 
        For every irreducible closed subset $Z\subseteq X$ such that the generic point $\eta _Z$ is contained in the regular locus of $X$ and any  $\lambda >0$ we have
        \[ 
            {\rm ord}_Z (\tau_+(\omega_X, \lambda \Delta)) > \lambda \cdot {\rm ord}_Z(\Delta )-{\rm codim}(Z,X).
        \]
    \end{lemma}
    \begin{proof} Since, as observed above, $\tau_+(\omega_X,  \lambda \cdot\Delta) \subseteq \Tr \pi_* \cO_Y(K_Y -\lfloor  \lambda \cdot \pi^* \Delta \rfloor)$, we may consider $\pi: Y\to X$ the blow up of $X$ along $Z$ and $E$ the exceptional divisor.
    We may restrict to an open subset of $X$ containing $\eta _Z$ so that we may assume that $Z$ is regular. If $c={\rm codim}(Z,X)$, then $K_{Y/X}=(c-1)E$ and $\lfloor \lambda \cdot \pi^* \Delta \rfloor \geq \lfloor \lambda \cdot {\rm ord}_Z(\Delta ) \rfloor E$. But then, using the fact that on this open subset we can assume $K_X=0$, we have 
    \[\Tr \pi_* \cO_Y(K_Y -\lfloor  \lambda \cdot \pi^* \Delta \rfloor)\subseteq \Tr \pi_* \cO_Y\left( ((c-1)-\lfloor \lambda \cdot {\rm ord}_Z(\Delta ) \rfloor )E\right)=\mathcal I _Z^{(p)},\]
    where $p=\lfloor \lambda \cdot {\rm ord}_Z(\Delta ) \rfloor -(c-1)>\lambda \cdot {\rm ord}_Z(\Delta )-c$, and the claim follows.
    \end{proof}

\subsection{$\bigplus$-test \emph{ideals}}
    In this subsection, we define $\bigplus$-test \emph{ideals} $\tau_+(\cO_X, \Delta) \subseteq \cO_X$ instead of submodules of $\omega_X$.  Indeed, by \autoref{lem.BasicFactsAboutTauAndAdj} \autoref{lem.BasicFactsAboutTauAndAdj.b}  that for $H \geq 0$ a Cartier divisor 
    \[
        \tau_+(\omega_X, B + H) = \tau_+(\omega_X, B) \otimes \cO_X(-H).
    \]
    Inspired by this, we make the following definition.

    \begin{definition}[$\tau_+(\omega_X)$ version 2]
        \label{def.Tau+V2}
        Suppose that $B$ is an arbitrary $\bQ$-divisor (not necessarily effective).  We define the \emph{$\bigplus$-test submodule} $\tau_+(\omega_X, B)$ to be $\tau_+(\omega_X, B+ H) \otimes \cO_X(H)$ where $H$ is a Cartier divisor such that $H + B\geq 0$.  

        Likewise we define $\adj_+^D(\omega_X(D), B) = \adj_+^D(\omega_X(D), B+H) \otimes \cO_X(H)$.
        It is straightforward to verify that these definitions are independent of the choice of $H$.
    \end{definition}

    Notice that if $\Delta$ is non-effective, we can have $\tau_+(\omega_X, \Delta) \not\subseteq \omega_X$.  We use this to obtain the following notion of test \emph{ideal}.

    \begin{definition}
        \label{def.Tau+Ideal}
        With $X$ as above, fix $\Delta \geq 0$ a $\bQ$-divisor.  Fix a canonical divisor $K_X$ and write $\omega_X = \cO_X(K_X)$. We then define the test ideal to be:
        \[
            \tau_+(\cO_X, \Delta) := \tau_+(\omega_X, K_X + \Delta).
        \]                
        if $\Delta = 0$ then we write $\tau_+(\cO_X)$ for $\tau_+(\cO_X, 0)$.  
        Likewise, if $D$ is reduced, we write:
        \[
            \adj_+^D(\cO_X, \Delta) = \adj^D_+(\omega_X(D), K_X + D + \Delta).
        \]
    \end{definition}
    \begin{remark}
        This notation conflicts slightly with the notation of \cite{MaSchwedeTuckerWaldronWitaszekAdjoint}.  There, in the case that $X = \Spec R$, $\adj^D_{R^+ \shortrightarrow (R/I_D)^+}(R, D+\Delta)$ was used to denote the object we are calling $\adj_+^D(\cO_X, \Delta)$.
    \end{remark}

   
    \begin{lemma}
        \label{lem.CleanGlobalGenerationOfTestIdeals}
        With $X$ and $\Delta$ as above, suppose that $\sM = \cO_X(M)$ is a sufficiently ample line bundle, then 
        $\myB^0(X, \Delta, \sM) \subseteq H^0(X, \sM)$ generates $\tau_+(\cO_X, \Delta) \otimes \sM$.

        Likewise, if $D$ is a reduced divisor then $\myB^0_D(X, D+\Delta, \sM)$ generates $\adj^D_+(\cO_X, \Delta) \otimes \sM$.  
    \end{lemma}
    \begin{proof}
        For the first statement, fix a Cartier divisor $H \geq 0$ so that $K_X + H \geq 0$ and choose $L$ sufficiently ample (so that $L - H$ is also sufficiently ample).  We know that 
        \[
            \tau_+(\cO_X, \Delta) \otimes \cO_X(L-H) = \tau_+(\omega_X, K_X + \Delta) \otimes \cO_X(L-H) = \tau_+(\omega_X, K_X + \Delta + H) \otimes \cO_X(L)
        \] 
        is globally generated by $\myB^0(X, K_X + \Delta + H, \cO_X(K_X + L))$.  
        
        Choose $M = L - H$, then by \autoref{lem.SimplePropertiesOfB0} \autoref{lem.SimplePropertiesOfB0.3} 
        \[            
                \myB^0(X, K_X + \Delta + H, \cO_X(K_X + L)) = \myB^0(X, \Delta, \cO_X(K_X + L - K_X - H))
                = \myB^0(X, \Delta, \cO_X(M)).            
        \]
        The first result follows.

        The second statement is similar.  We have that 
        \[            
            \adj^D_+(\cO_X, \Delta) \otimes \cO_X(L-H) = \adj_+^D(\omega_X(D), K_X + D + \Delta + H) \otimes \cO_X(L)        
        \] 
        is globally generated by $\myB^0_D(X, K_X + 2D + \Delta + H, \cO_X(K_X + D + L))$.  But setting $M = L - H$, we have
        \[
            \myB^0_D(X, K_X + 2D + \Delta + H, \cO_X(K_X + D + L)) = \myB^0_D(X, D+\Delta, \cO_X(M))
        \]
        which completes the proof.
    \end{proof}

    \begin{lemma}
        With notation as in \autoref{def.Tau+Ideal} and assuming that $\Delta \geq 0$, we have that $\tau_+(\cO_X, \Delta)$ is an ideal sheaf.  Likewise, $\adj_+^D(\cO_X, \Delta) \subseteq \cO_X$.
    \end{lemma}
    \begin{proof}
        By \autoref{lem.CleanGlobalGenerationOfTestIdeals} we have that $\myB^0(X, \Delta, \sM) \subseteq H^0(X, \sM)$ globally generates $\tau_+(\cO_X, \Delta) \otimes \sM \subseteq \sM$.  In particular, $\tau_+(\cO_X, \Delta) \subseteq \cO_X$.  The statement for $\adj_+^D(\cO_X, \Delta)$ is similar.
    \end{proof}

    \begin{remark}[The quasi-Gorenstein case]
        If $X$ is quasi-Gorenstein (meaning that $K_X$ is Cartier), then it follows from our definition that $\tau_+(\cO_X, \Delta) = \tau_+(\omega_X, K_X + \Delta) = \tau_+(\omega_X, \Delta) \otimes \cO_X(-K_X)$.  Hence 
        \[
            \tau_+(\cO_X, \Delta) \otimes \omega_X = \tau_+(\omega_X, \Delta).  
        \]
    \end{remark}

    We have the following containments.
    \begin{lemma}
        \label{lem.SimpleAdjointTestContainmentsForIdeals}
        With notation as above, if $\Delta \geq 0$ is a $\bQ$-divisor and $D$ is a reduced divisor, then 
        \[
            \tau_+(\cO_X, D+\Delta) \subseteq \adj_+^D(\cO_X, \Delta) \subseteq \tau_+(\cO_X, (1-\epsilon)D + \Delta)
        \]
        for all $0 < \epsilon \leq 1$.  
    \end{lemma}
    \begin{proof}
        By \autoref{lem.BasicFactsAboutTauAndAdj} \autoref{lem.BasicFactsAboutTauAndAdj.c} replacing $B$ by $K_X + B$, we have that 
        \[
            \tau_+(\omega_X, K_X + D+B) \subseteq \adj_+^D(\omega_X(D), K_X + D+B) \subseteq \tau_+(\omega_X, K_X + (1-\epsilon)D + B).
        \]            
        By definition, this is 
        \[
            \tau_+(\cO_X, D+B) \subseteq \adj_+^D(\cO_X, B) \subseteq \tau_+(\cO_X, (1-\epsilon)D + B).
        \]
    \end{proof}

    We point out the following transformation rule for finite maps.
    \begin{theorem}[Transformation rules for finite maps]
        \label{thm.TransofmrationRulesUnderFiniteMaps}
        With notation as above, suppose that $f : Y \to X$ is a finite dominant map between normal integral schemes.  Then for any $B \geq 0$, we have that 
        \[
            \Tr \big( f_* \tau_+(\omega_Y, f^* B)\big) = \tau_+(\omega_X, B).
        \]
        Furthermore, if $\Delta \geq 0$ and $f$ is generically separable with ramification divisor $\Ram$, then we have 
        \[
            \Tr \big( f_* \tau_+(\cO_Y, f^* \Delta - \Ram)) = \tau_+(\cO_X, \Delta).
        \]
    \end{theorem}
    \begin{proof}
        By the inverse limit description of $\myB^0$ found in \autoref{lem.BasicPropertiesOfB^0AndB^0_D}, we know that         
        \[
            \myB^0(X, B, \omega_X \otimes \sL) = \Tr \big( \myB^0(Y, f^* B, \omega_Y \otimes f^* \sL) \big).  
        \]
        Choosing $\sL$ sufficiently ample so that $f^* \sL$ is also sufficiently ample, we have using \autoref{prop.B0ForBigiIsGlobalSections} that 
        \[
            H^0(X, \tau_+(\omega_X, B) \otimes \sL) = \Tr \big( H^0(Y, \tau_+(\omega_Y, f^* B) \otimes f^* \sL) \big).
        \]
        Now, replacing $\sL$ with a power if necessary, we have that $(f_* \tau_+(\omega_Y, f^* B)) \otimes \sL$ is also globally generated as a $\cO_X$-module.  
        Its image under the following composition 
        \[
            f_* \big( \tau_+(\omega_Y, f^* B) \otimes f^* \sL) \hookrightarrow f_* \big( \omega_Y \otimes f^* \sL \big) \xrightarrow{\Tr} \omega_X \otimes \sL
        \]
        is thus globally generated.  Furthermore, if $J \otimes \sL \subseteq \omega_X \otimes \sL$ is that image, we have a surjection of global sections (again replacing $\sL$ by a power if necessary)
        \[
            H^0(X, f_* \big( \tau_+(\omega_Y, f^* B) \otimes f^* \sL)) \to H^0(X, J \otimes \sL).
        \]
        But the image of those global sections is $\myB^0(X, B, \omega_X \otimes \sL) = H^0(X, \tau_+(\omega_X, B) \otimes \sL)$.
        Thus $J \otimes \sL = \tau_+(\omega_X, B) \otimes \sL$ and the first statement is proven.  
        
        The second statement follows from the first via the techniques we have already used (or by mimicking the argument above), once one notes that $K_Y = f^* K_X + \Ram$.  
    \end{proof}

    \'Etale extensions of the base also behave well.

    \begin{proposition}
        \label{prop.Tau+BaseChangesAlongEtaleMapsOfTheBase}
        With notation as above, suppose that $H^0(X, \cO_X) = R \subseteq S$ is a finite \'etale extension.  Then 
        \[
            \tau_+(\omega_X, B) \otimes_R S = \tau_+(\omega_{X_S}, B_S)
        \]
        where $X_S, B_S$ denote the base change.
    \end{proposition}
    \begin{proof}
        This is a direct consequence of \autoref{prop.FiniteEtalExtensionOfRAndB^0}.  
    \end{proof}

\begin{theorem}[Adjunction and inversion thereof]
    \label{thm.AdjunctionForAdjTau}
    With $X$ as above, suppose that $K_X + D + \Delta$ is $\bQ$-Cartier where $D$ is reduced, $\Delta \geq 0$, and $D$ and $\Delta$ have no common components.  Then
    \[
        \adj^D_+(\cO_X, \Delta) \cdot \cO_{D^{\Norm}} = \tau_+(\cO_{D^{\Norm}}, \Diff_{D^{\Norm}}(D+\Delta)).
    \]
\end{theorem}
\begin{proof}
    Choose $L = K_X + M$ with $\sL = \cO_X(L)$ sufficiently ample on $X$ whose restriction is also sufficiently ample on $D^{\Norm}$.  We have a surjection from \autoref{thm.LiftingSections}:
    \[
        \myB^0_D(X, D+\Delta, \sL) \twoheadrightarrow \myB^0(D^{\Norm}, \Diff_{D^{\Norm}}(\Delta + D), \sL|_{D^{\Norm}}).
    \]
    By \autoref{lem.CleanGlobalGenerationOfTestIdeals}, the left side generates $\adj^D_+(\cO_X, \Delta)\otimes \sL$ as a $\cO_X$-module, and the right side generates $\tau_+(\cO_{D^{\Norm}}, \Diff_{D^{\Norm}}(D+\Delta))\otimes \sL|_{D^N}$ as a $\cO_{D^{\Norm}}$-module.  In particular the image of $\adj^D_+(\cO_X, \Delta)$ under the map $\cO_X \to \cO_{D^{\Norm}}$ generates the ideal $\tau_+(\cO_{D^{\Norm}}, \Diff_{D^{\Norm}}(D+\Delta))$.  This is what we wanted to show.
\end{proof}

\subsection{Test ideals on the nonsingular locus}

    \begin{proposition}
        \label{prop.tauAgreesWithOmegaOnSNCLocus}
        With $(X, B)$ as above, let $U \subseteq X$ denote the locus where $X$ is regular, $B$ is SNC and $\lfloor B \rfloor = 0$.  Then 
        \[
            \tau_+(\omega_X, B)|_U = \omega_U
        \]
        and
        \[
            \tau_+(\cO_X, B)|_U = \cO_U.
        \]
    \end{proposition}
    \begin{proof}
        Choose a (not necessarily closed) point $x \in U$.   We proceed by induction on the codimension/height of $x$ as a point of $X$.  If $x$ is the generic point of $X$, there is nothing to prove since these sheaves are nonzero by \autoref{lem.TauIsNonzero}. 
        
        Since $x$ is nonsingular, we may pick $K_X$ that does not contain $x$ in its support.  Choosing $E \geq 0$ a Weil divisor also not containing $x$ in its support so that $K_X + E \geq 0$ is Cartier, we have that 
        \[
            \tau_+(\omega_X, B) \supseteq \tau_+(\omega_X, B + K_X + E) = \tau_+(\cO_X, B + E) \subseteq \tau_+(\cO_X, B).
        \]  
        It then suffices to show that $\tau_+(\cO_X, B+E)$ agrees with $\cO_X$ at $x$ to show both statements.  We then replace $B$ by $B+E$ and forget about $E$ (note that $\lfloor B \rfloor$ may no longer equal zero, but it is zero in a neighborhood of $x$).
        
        Choose $D$ a prime divisor passing through $x$ that is nonsingular at $x$ so that $B \vee D$ is SNC at $x$.  There are two ways to do this, if $x \notin \Supp B$ then choose $D$ arbitrarily but nonsingular at $x$.  Otherwise, choose $D$ to be one of the components of $\Supp B$ at $x$.  Either way, write $B \vee D = B' + D$ where $B'$ does not have $D$ in its support.  Finally, choose $B'' \geq 0$ a $\bQ$-divisor not containing $x$ in its support so that $K_X + D + B' + B''$ is $\bQ$-Cartier.  Since 
        \[
            (1 - \epsilon)D + B' + B'' \geq B
        \]
        for $\epsilon \gg 0$, we have that 
        \[
            \adj^D_+(\cO_X, B'+B'') \subseteq \tau_+(\cO_X, (1 - \epsilon)D + B' + B'') \subseteq \tau_+(\cO_X, B)
        \]
        where the first containment is \autoref{lem.SimpleAdjointTestContainmentsForIdeals}.
        Thus it suffices to show that $\adj^D_+(\cO_X, B'+B'')_x = \cO_{X,x}$.  

     Now, by \autoref{thm.AdjunctionForAdjTau}
        \[
            \adj^D_+(\cO_X, B' + B'') \cdot \cO_{D^{\Norm}} = \tau_+(\cO_{D^{\Norm}}, \Diff_{D^{\Norm}}(D+B'+B''))
        \]
        so it suffices to show the theorem for the right side since $D$ is normal at $x$.  Since $D$ is regular at $x$ and $\Diff_{D^{\Norm}}(D+B'+B'')$ is simple normal crossings at $x$, we are done by induction.
    \end{proof}

\subsection{A brief comparison with characteristic $p > 0$}
        
        Suppose briefly that $R$ is an $F$-finite field\footnote{We expect everything works for more general complete local Noetherian bases, but we restrict to this case here to keep the arguments short.} of characteristic $p > 0$ and $X \to \Spec R$ is projective where $X$ is normal and integral.  Everything we have defined so far still makes sense.  Indeed, in the case where $B$ is $\bQ$-Cartier and $M$ is Cartier, we have that $\myB^0(X, \Delta, \cO_X(K_X + M))$ agrees with $T^0(X, M - \Delta)$ as defined in \cite[Section 6]{BlickleSchwedeTuckerTestAlterations}.  In that case, we find that 
        \[
            \myB^0(X, B, \cO_X(K_X+ M)) = \Image\Big(H^0(Y, \cO_Y(K_Y - f^*B + f^*M)) \to H^0(X, \cO_X(K_X + M))\Big)
        \]
        for a sufficiently large $f : Y \to X$ finite map or alteration.  This stabilization does not occur in mixed characteristic, see \cite[Example 4.14]{BMPSTWW-MMP}.
        On the other hand, for such a sufficiently large finite map, we also have that 
        \[
            \tau(\omega_X, B) = \Image\Big( f_* \cO_Y(K_Y - f^*B) \to \omega_X \Big)
        \]
        where $\tau(\omega_X, B)$ is the test submodule, see for instance \cite[Definition 2.33 and the proof of Theorem 4.6]{BlickleSchwedeTuckerTestAlterations}.  
        
        \begin{theorem}
            \label{thm.AgreesWithCharPTestIdeal}
            With notation as above, if $B$ is $\bQ$-Cartier, we have that 
            \[
                \tau_+(\omega_X, B) = \tau(\omega_X, B).
            \]
            Hence, if $\Delta \geq 0$ is a $\bQ$-divisor such that $K_X + \Delta$ is $\bQ$-Cartier, then we have that 
            \[
                \tau_+(\cO_X, \Delta) = \tau(\cO_X, \Delta).
            \]
        \end{theorem}
        \begin{proof}
            Choose $g : X' \to X$ with $g^* B$ Cartier.  Since $\Tr(g_* \tau_+(\omega_{X'}, g^* B)) = \tau_+(\omega_X, B)$ by \autoref{thm.TransofmrationRulesUnderFiniteMaps} and $\Tr(g_* \tau(\omega_{X'}, g^* B)) = \tau(\omega_X, B)$ by \cite[Proposition 4.4]{BlickleSchwedeTuckerTestAlterations}, we may replace $X$ by $X'$ and so assume $B$ is Cartier.  But now 
            \[
                \tau_+(\omega_X, B) = \tau_+(\omega_X) \otimes \cO_X(-B) \text{ and } \tau(\omega_X, B) = \tau(\omega_X) \otimes \cO_X(-B)
            \]
            and so we may assume that $B = 0$.

            Let $F : X \to X$ be the absolute Frobenius.  Then $\Tr_F (\tau_+(\omega_X)) = \tau_+(\omega_X)$.  But $\tau(\omega_X)$ is the smallest nonzero module with this property, so we have that 
            \[
                \tau(\omega_X) \subseteq \tau_+(\omega_X).
            \]
            
            For the reverse inclusion notice that for a sufficiently large $f : Y \to X$ with 
            \[
                \tau(\omega_X) = \Image\Big( f_* \omega_Y \to \omega_X \Big)
            \]
            we have that
            \[
                \myB^0(X, B, \cO_X(K_X+ M)) = \Image\Big(H^0(Y, \cO_Y(K_Y + f^*M)) \to H^0(X, \cO_X(K_X + M))\Big)
            \]
            is contained in $H^0(X, \tau(\omega_X) \otimes M)$.  Hence, if we choose $M$ ample enough so that $\myB^0(X, 0, \cO_X(K_X + M))$ globally generates $\tau_+(\omega_X) \otimes \cO_X(M)$ then we obtain that
            \[
                \tau_+(\omega_X) \subseteq \tau(\omega_X).
            \]
            This finishes the proof of the first statement.

            The second statement follows from the first since we have 
            $\tau_+(\cO_X, \Delta) = \tau_+(\omega_X, K_X + \Delta)$ and $\tau(\cO_X, \Delta) = \tau(\omega_X, K_X + \Delta)$.            
        \end{proof}

\section{Passing to affine charts and localization}
\label{sec.Tau+OnAffineCharts}

    Our goal in this section is to show that $\tau_+$ is actually well-defined affine locally.  First we point out a transformation rule for birational maps.  The same containment also holds in characteristic $p > 0$.


    \begin{proposition}
        \label{prop.ContainmentBirationalMaps}
        Suppose that $X$ is normal and integral, projective over a Noetherian complete local ring $(R, \fram)$ of mixed characteristic.  Suppose additionally that $B$ is $\bQ$-Cartier on $X$ and that $\pi : Y \to X$ is any projective birational map.  Then 
        \[
             \tau_+(\omega_X, B) \subseteq \pi_* \tau_+(\omega_Y, \pi^* B).
        \]
    \end{proposition}
    \begin{proof}
        Fix $\sL = \cO_X(L)$ sufficiently ample on $X$.  Then we know that 
        \[
            \myB^0(X, B, \cO_X(K_X + L)) = \myB^0(Y, \pi^* B, \cO_Y(K_Y + \pi^*L))  
        \]
        by \autoref{prop.B0StableBirational}.  Now, $\myB^0(Y, \pi^* B, \cO_Y(K_Y + \pi^*L))$ generates a $\cO_Y$-subsheaf $J \otimes \pi^* \sL \subseteq \cO_Y(K_Y + \pi^* L)$.  
        Note that $\pi_* J \supseteq \tau_+(\omega_X, B)$ since the global generators of the $\cO_Y$-sheaf $J \otimes \pi^* \sL$ also generate the $\cO_X$-sheaf $\tau_+(\omega_X, B) \otimes \sL$.

        If we twist by a sufficiently ample $\sM = \cO_Y(M)$ on $Y$, then since we have the multiplication map of \autoref{lemma.EasyInclusionForB0ByMultiplication},
        \[
            \myB^0(Y, \pi^* B, \cO_Y(K_Y + \pi^* L)) \otimes H^0(Y, \cO_Y(M)) \to  \myB^0(Y, \pi^* B, \cO_Y(K_Y + \pi^* L + M)),
        \]
        we see that $J \otimes \cO_Y(\pi^* L + M) \subseteq \tau_+(\omega_Y, \pi^* B) \otimes \cO_Y(\pi^* L + M)$ and so $J \subseteq \tau_+(\omega_Y, \pi^* B)$.  Hence we have that 
        \[
            \tau_+(\omega_X, B) \subseteq \pi_* J \subseteq \pi_* \tau_+(\omega_Y, \pi^* B)
        \]
        as desired.
    \end{proof}

    \begin{remark}
        Note cannot expect equality here.  Indeed, \cite[Example 4.14]{CarvajalRojasMaPolstraSchwedeTucker} provides a counter example.  Explicitly, the singularity $R = W(k)\llbracket x,y \rrbracket/(p^2+y^3+z^5)$ where $k$ is an algebraically closed field of characteristic $p > 0$ is a rational Gorenstein singularity (see \cite[Proposition 8.1]{LipmanRationalSingularities}) and so if one resolves the singularity $\pi : Y \to X = \Spec R$ we see that $\tau_+(\omega_Y) = \omega_Y$ and since $R$ has rational singularities, $\pi_* \omega_Y = \omega_X$.  On the other hand, by \cite[Example 4.14]{CarvajalRojasMaPolstraSchwedeTucker} we see that $R$ is not a splinter and it follows that $\tau_+(\omega_R) \subsetneq \omega_R$.
    \end{remark}

    
    \begin{theorem}
        \label{thm.Tau+LocallyIndependentOfBlowup}
        Suppose that $(X, B)$ is a pair as above where $B$ is $\bQ$-Cartier and $U \subseteq X$ is an open set.  
        Now suppose that $g : X' \to X$ is a projective birational map that is an isomorphism over $U$ with $U' = g^{-1}(U)$.  Then we have that 
        \[
            \tau_+(\omega_{X'}, g^* B)|_{U'} = \tau_+(\omega_X, B)|_U.
        \]  
    \end{theorem}
    \begin{proof}
        Fix $\sA = \cO_X(A)$ a sufficiently ample line bundle on $X$.  Notice that if $U_i$ is an open cover of $U$, and we prove the result for each $U_i$, then we have also proven it for $U$.  Hence we may replace $U$ by a smaller open set whenever it is helpful. In particular, we may assume that $U$ is affine and that $U$ is the complement of some ample divisor $H \sim A$ (possibly replacing $A$ by a multiple).  
        
        \begin{claim}
            Shrinking $U$ if necessary and making $\sA$ more ample if necessary, we may assume that there exists an effective Cartier divisor $E$ on $X'$ such that $E \cap U' = \emptyset$ and $\sA'=g^*\sA(-E) = \cO_{X'}(g^* A - E)$ is ample.
        \end{claim}
        \begin{proof}[Proof of claim]
            We know that $g$ is the blowup of some ideal sheaf $\sJ$ \cite[II 7.17]{Hartshorne}.  Fix a point $x \in U$.  If $\sJ$ does not vanish at $x$ we are done, since we can write $\sJ \cdot \cO_{X'} = \cO_{X'}(-E)$ (replacing $\sA$ with a power if necessary).  If it does vanish at $x$, we must replace $\sJ$ by an ideal sheaf, with the same blowup, that does not vanish at $x$.  

            We know that $\sJ$ is locally principal at $x$ since its blowup is an isomorphism at $x$.  We may assume that $\sJ \otimes \sA$ is globally generated (replacing $\sA$ by a power).  Choose $s \in H^0(X, \sJ \otimes \sA) \subseteq H^0(X, \sA)$ that generates $\sJ \otimes \sA$ at $x$.  Then if $D \sim A$ is the Cartier divisor corresponding to $s \in H^0(X, \sA)$, we have that the \emph{fractional} ideal $\sJ \otimes \cO_X(D)$ agrees with $\cO_X$ in a neighborhood of $x$.  Shrinking $U$ if necessary (which implicitly increases $\sA$ since $U$ is the complement of some divisor $H$ linearly equivalent to $A$), we can assume that $(\sJ \otimes \cO_X(D))|_U = \cO_U$.  It follows that $\sJ \otimes \sO_X(D - m H) \subseteq \cO_X$ for $m \gg 0$, it agrees with $\cO_X$ at $x$, and has the same blowup as $\sJ$.  This proves the claim.
        \end{proof}
            There is an effective Cartier divisor $F$ on $X$ such that $F\cap U=\emptyset$ and $g^*F\geq E$.  
            Now take $F = mH$ for some $m \gg 0$.
%
        
        Since  $\tau _+(\omega _X,B+kF)=\tau_+ (\omega _X,B)\otimes \mathcal O _X(-kF)$ by \autoref{lem.BasicFactsAboutTauAndAdj}, then $\tau _+(\omega _X,B+kF)|_U\cong \tau _+(\omega _X,B)|_U$. Similarly, $\tau_+ (\omega _{X'},g^*(B+kF))=\tau_+ (\omega _{X'},g^*B)\otimes \mathcal O _{X'}(-kg^*F)$ so that $\tau _+(\omega _{X'},g^*(B+kF))|_{U'}\cong \tau _+(\omega _{X'},g^*B)|_{U'}$.
        By \autoref{prop.B0StableBirational}, (see also \cite[Lemma 4.19]{BMPSTWW-MMP}),
        we have
        \begin{equation}
            \myB^0(X,B,\omega_X \otimes \sA^i )=\myB^0(X',g^* B,\omega_{X'} \otimes g^* \sA^i)\supset \myB ^0(X',g^* B,  \omega_{X'} \otimes \sA'^i   )
        \end{equation}
        where the last inclusion is induced by the natural map $\sA'^i\cong \mathcal O _{X'}(i(g^*A-E))
        \subseteq  \mathcal O _{X'}(ig^*A)$.
        Therefore, taking $i \gg 0$ we obtain that 
        \[ 
            \tau _+(\omega _{X'},g^* B)|_{U'}\subseteq \tau _+(\omega _X,B)|_U.
        \]
        The reverse inclusion follows from \autoref{prop.ContainmentBirationalMaps}.
    \end{proof}

    \begin{lemma}
        \label{lem.DivisorOnAnOpenQCartierCompactification}
        Suppose $X$ is a normal projective integral scheme over a Noetherian ring $R$ and $U \subseteq X$ is an affine open set, the complement of an ample effective Cartier divisor $H$.  Suppose that $B_U \geq 0$ is a $\bQ$-trivial\footnote{Meaning that $nB_U \sim 0$ for some $n > 0$.  This is not a particularly restrictive condition on an affine open set, since one may always restrict to an affine open cover trivializing any $\bQ$-Cartier divisor.} divisor.  Then there exists a $\bQ$-Cartier divisor $B$ on $X$ such that $B|_U = B_U$.
    \end{lemma}
    \begin{proof}
        Write $D_U = nB_U = \Div(f)$ for some $f \in \Gamma(U, \cO_U)$.  Then viewing $f \in K(X)$, it defines a rational divisor $D'$ on $X$ with $D'|_U = nB_U$.  However, $D'$ is possibly not effective but $D' + mH$ is effective for some $m > 0$ and so we set $B = {1 \over n}(D' + mH)$.
    \end{proof}

    \begin{definition}
        Suppose that $(R, \fram)$ is a complete local ring with positive characteristic residue field.  Suppose further that $T$ is a finite type $R$-algebra.  Let $U = \Spec T$ and suppose that $B_U \geq 0$ is a $\bQ$-trivial divisor on $U$.  Take a projective scheme $X$ over $\Spec R$ where $\Spec T = U = X \setminus V(H)$ is an affine open subset, the complement of a very ample divisor.  Suppose that $B \geq 0$ is  a $\bQ$-Cartier-divisor on $X$ such that $B|_U = B_U$.        
        
        Then define 
        \[
            \tau_+(\omega_{U}, B_U) = \tau_+(\omega_X, B)|_U.
        \]
        We show that this definition is independent of the choices  of $X$ and $B$.
    \end{definition}
    \begin{proposition}
        \label{prop.Tau+IsIndependentOfCompactifications}
        The sheaf $\tau_+(\omega_{U}, B_U)$ is independent of the compactifications $X$ and $B$.        
    \end{proposition}
    \begin{proof}
        Suppose first that $B_1, B_2$ are two effective $\bQ$-Cartier divisors on $X$ that restrict to $B_U$ on $U$.  Let $\pi : Y \to X$ be a finite cover from a normal integral scheme so that both $\pi^* B_1$ and $\pi^* B_2$ are Cartier.  Let $V = \pi^{-1}(U)$.  Since $\Tr : \pi_* \omega_V \to \omega_X$ satisfies $\Tr(\tau_+(\omega_Y, \pi^* B_i)) = \tau_+(\omega_X, B_i)$ for $i = 1,2$ by \autoref{thm.TransofmrationRulesUnderFiniteMaps}, it suffices to prove the result on $V \subseteq Y$.  In particular, replacing $X$ by $Y$, we may assume that $B_U, B_1, B_2$ are all Cartier.  Furthermore, we may assume that $B_2 \geq B_1$ with $B_2 = B_1 + G$ where $G|_U = 0$.  Then we have that 
        \[
            \tau_+(\omega_X, B_2)|_U = \big(\tau_+(\omega_X, B_1) \otimes \cO_X(-G)\big)|_U = \tau_+(\omega_X, B_1)|_U
        \]
        and we have thus shown that our definition is independent of the choice of $B$.


        Suppose now that $X_1, X_2$ are two compactifications of $U$.  Passing to finite covers as above, we may assume that $B_U, B_1, B_2$ are all Cartier divisors.  Since $\tau_+(\omega_{X_i}, B_i) = \tau_+(\omega_{X_i}) \otimes \cO_X(-B_i)$, it suffices to show that $\tau_+(\omega_{X_1})|_U = \tau_+(\omega_{X_2})|_U$.  We then have that $X_1 \dashrightarrow X_2$ is a rational map between projective schemes over a Noetherian local ring $R$.  
            We resolve the indeterminacies of this rational map.  That is, we have 
            \[
                \xymatrix{
                    & X_3 \ar[dl] \ar[dr] & \\
                    X_1 \ar@{.>}[rr]_{\mu} & &  X_2
                }
            \]
            where both diagonal arrows are blowups and are isomorphisms over $U$.  But now we are done by \autoref{thm.Tau+LocallyIndependentOfBlowup}.  
    \end{proof}

    \begin{corollary}
        Suppose $X$ is a quasi-projective normal integral scheme over a complete local ring ring $(R, \fram)$.  Further suppose that $B \geq 0$ on $X$ is a $\bQ$-Cartier divisor.  Then there exists a well defined $\tau_+(\omega_X, B)$ which restricts to $\tau_+(\omega_{X'}, B|_{X'})$ for any open $X' \subseteq X$.  In particular, if $X = \Spec T$ is affine, then we can define $\tau_{+}(\omega_X, B)$ commuting with localization at a single element.  
    \end{corollary}
    \begin{proof}
        Since $\tau_+(\omega_X, B)$ is a subsheaf of $\omega_X$, it suffices to check this on a sufficiently fine affine cover.   But now the result follows from \autoref{prop.Tau+IsIndependentOfCompactifications}.  
    \end{proof}

    \begin{lemma}
        \label{lem.QuasiProjectiveTwistedTauOmega}
        Suppose $X$ is quasi-projective over a complete local ring $(R, \fram)$ with residual characteristic $p > 0$ and $B \geq 0$.  Suppose further that $G \geq 0$ is a Cartier divisor.  Then 
        \[
            \tau_+(\omega_X, B + G) = \tau_+(\omega_X, B) \otimes \cO_X(-G).
        \]
    \end{lemma}
    \begin{proof}
        This may be checked locally and so we may assume that $G$ is a principal divisor ($G \sim 0$).  Thus we may compactify to $\overline{X}$, $\overline{B}$ and $\overline{G}$ where $\overline{B}$ is $\bQ$-Cartier and $\overline{G}$ is Cartier.  Now the result follows from \autoref{lem.BasicFactsAboutTauAndAdj} \autoref{lem.BasicFactsAboutTauAndAdj.b}.
    \end{proof}

    \begin{definition}
        Suppose that $X$ is quasi-projective over a complete local ring $(R, \fram)$.  Suppose $\Delta \geq 0$ is a $\bQ$-divisor so that $K_X + \Delta$ is $\bQ$-Cartier.  Then we define 
        \[
            \tau_+(\cO_X, \Delta) = \tau_+(\omega_X, K_X + \Delta + G) \otimes \cO_X(G)
        \]
        where $G$ is a Cartier divisor so that $K_X + \Delta + G \geq 0$.    It follows from \autoref{lem.QuasiProjectiveTwistedTauOmega} that this is independent of the choice of $G$ in the definition.
    \end{definition}

    \begin{remark}
        Suppose that $T$ is a normal domain essentially of finite type over a complete local ring $(R, \fram)$ with residue field of characteristic $p > 0$.  Then for any $B \geq 0$, a $\bQ$-Cartier divisor on $\Spec T$, we can define $\tau_+(\omega_T, B)$ whose formation commutes with localization.  The point is we can unlocalize to a finite type $R$-algebra domain where $B$ is still $\bQ$-Cartier.  Likewise we can define $\tau_+(T, \Delta)$ where $K_T + \Delta$ is $\bQ$-Cartier.
    \end{remark}

    \begin{remark}
        Suppose for instance that $X = \Spec R$.  Then $\tau_+(\cO_X, B)$ is nothing other than the test ideal $\tau_{\widehat{R^+}}(R, B)$ as defined in \cite{MaSchwedeSingularitiesMixedCharBCM}.  We obtain a theory of localization then, but it is nothing more than localizing this ideal $\tau_{\widehat{R^+}}(R, B)$.  
    \end{remark}


\section{Effective global generation of $\tau_+$ and a Skoda-type theorem}

    Fix $X$ a normal integral scheme projective over a complete local Noetherian domain $(R, \fram)$.  Suppose $H^0(X, \cO_X) = R$.  Suppose that $X_{\fram}$ is the special fiber and set $n_0 = \dim X_{\fram}$.

    \begin{theorem}[Effective global generation]\label{t-eff}
        With notation as above, if $\sA = \cO_X(A)$ is a globally generated ample line bundle, $\Gamma \geq 0$ is a $\bQ$-divisor and $L$ is a divisor such that $L - \Gamma$ is big and semi-ample.  Then 
        \[
            \tau_+(\omega_X, \Gamma) \otimes \cO_X(nA + L)
        \]
        is globally generated by $\myB^0(X, \Gamma, \omega_X \otimes \cO_X(nA + L))$ for all $n \geq n_0 = \dim X_{\fram}$.
    \end{theorem}

    Our proof is inspired by similar arguments in characteristic $p > 0$ found in \cite{SchwedeTuckerTestIdealsNonPrincipal} which themselves are based upon arguments for multiplier ideals found in \cite{LazarsfeldPositivity2}, also see \cite[Section 1]{EinLazarsfeldGeometricEffectiveNullstellensatz}.

    \begin{proof}        
        Set $R = H^0(X, \cO_X)$.  If the residue field $k = R/\fram$ is infinite, we may fix global generators $y_1, \dots, y_{n+1} \in H^0(X, \sA)$ where $n \geq n_0$ and $n_0$ is the dimension of the fiber.  The point is that sections in $H^0(X, \sA)$ globally generate $\sA$ if and only if their images globally generate $\sA|_{X_{\fram}}$ which is projective of dimension $n$ over an infinite field.  

        If $k$ is not infinite, we can find $n+1$ generating sections $y_1, \dots, y_{n+1}$ in $H^0(X, \sA) \otimes_R S = H^0(X_S, \sA_S)$ for some finite \'etale extension $R \subseteq S$ (inducing a finite separable extension of $k$).  Since global generation of a sheaf is unaffected by faithfully flat base change, by using \autoref{prop.FiniteEtalExtensionOfRAndB^0} and \autoref{prop.Tau+BaseChangesAlongEtaleMapsOfTheBase}, we may replace $R$ by $S$, $X$ by $X_S = X \times_{\Spec R} \Spec S$, etc. and so either way we have our desired global generators $y_i$.

        Consider the Skoda/Koszul complex determined by the $y_i$:
        \[
            0 \to \sF_{n+1} \to \sF_{n} \to \dots \to \sF_1 \to \sF_0 \to 0
        \]
        where $\sF_i = \cO_X(-iA)^{\oplus {n+1 \choose i}}$.  Since $A$ is a Cartier divisor, this complex is exact, see \cite[Proposition 1.6.5(c)]{BrunsHerzog}.  Let $X^+$ denote the absolute integral closure of $X$ as in \autoref{subsec.VanishingInMixedChar},  and use $\rho : X^+ \to X$ to denote the canonical map.  We pull back our Skoda/Koszul sequence to $X^+$.  Twisting by the line bundle $\cO_{X^+}(\rho^*((n+1)A + L - \Gamma))$ and applying the functor $\sHom_{\cO_{X^+}}(\bullet,  \cO_{X^+})$ we obtain
        \[
            0 \leftarrow \sG_{n+1} \leftarrow \sG_{n} \leftarrow \dots \leftarrow \sG_1 \leftarrow \sG_0 \leftarrow 0
        \]
        where $\sG_i = \cO_{X^+}(\rho^*(-(n+1-i)A - L +\Gamma))^{\oplus {n+1 \choose i}}$.    This is still exact since our original complex can be viewed as a locally-free resolution of a locally free sheaf.  We can also do the same thing on $X$, twisting by $(n+1)A + L$, dualizing, and setting $\sK_i = \cO_X(-(n+1-i)A - L)$.  Notice that since $\Gamma \geq 0$, we have canonical maps
        \[
            \sK_i \to \rho_* \sG_i
        \]
        for each $i$.

        Finally, if $d = \dim X$, taking top local cohomology  yields a diagram:
        \[
            {
                \xymatrix@R=12pt{
                    0 & \ar[l] H^{d}_{\fram}(\sG_{n+1}) & \ar[l] \ldots & \ar[l] H^{d}_{\fram}(\sG_2) & \ar[l]  H^{d}_{\fram} (\sG_1) & \ar@{_{(}->}[l]_-{\alpha} H^{d}_{\fram} (\sG_{0}) & \ar[l] 0\\
                    &  H^{d}_{\fram}(\sK_{n+1}) \ar[u] & \ar[l] \ldots & \ar[l] H^{d}_{\fram} (\sK_2) \ar[u] & \ar[l] H^{d}_{\fram}(\sK_1) \ar[u]_{j_1} & \ar[l] H^{d}_{\fram}(\sK_0) \ar[u]_{j_0} &  \\
                }    
            }
        \]
        Here we are abusing and condensing notation and writing $H^d_{\fram}(\sG_i)$ instead of $H^d_{\fram}(R \Gamma(X, \rho_* \sG_i))$.  The lower local cohomologies $H^{i}_{\fram}(\sG_j)$ for $i < d$ are all zero by Bhatt's vanishing theorem \autoref{BhattVanishing} \cite{BhattAbsoluteIntegralClosure,BMPSTWW-MMP}, so the top row is in fact exact and in particular the map labeled $\alpha$ is injective.  The bottom row is not necessarily exact though.
         We now provide a less condensed view of the square involving $\alpha$, $j_0$, and $j_1$.
        \[
            \xymatrix{
                \bigoplus_{n+1} H^{d}_{\fram} (\myR \Gamma(X^+, \cO_{X^+}(\rho^*(-nA - L +\Gamma)))) 
                & \ar@{_{(}->}[l]_-{\alpha} H^{d}_{\fram} (\myR \Gamma(X^+, \cO_{X^+}(\rho^*(-(n+1)A - L +\Gamma)))) \\
                \bigoplus_{n+1} H^{d}_{\fram} (\myR \Gamma(X, \cO_{X}(-nA - L)))  \ar[u]_{j_1} 
                & \ar[l] H^{d}_{\fram}(\myR \Gamma(X, \cO_{X}(-(n+1)A - L))) \ar[u]_{j_0}
            .}
        \]

        Applying Matlis duality $\bullet^{\vee} = \Hom(\bullet, E)$ where $E = E_{R/\fram}$ is the injective hull of the residue field of $R$, we see that
        \[
            \oplus_{n+1} \myB^0(X, \Gamma, \cO_X(K_X + nA + L)) = \Image(j_1)^{\vee} \twoheadrightarrow \Image(j_0)^{\vee} = \myB^0(X, \Gamma, \cO_X(K_X + (n+1)A + L))
        \]
        surjects.  In other words, the multiplication map
        \[
            \myB^0(X, \Gamma, \cO_X(K_X + nA + L)) \otimes H^0(X, \cO_X(A)) \to \myB^0(X, \Gamma, \cO_X(K_X + (n+1)A + L))
        \]  
        surjects.
        In fact, running through the same argument, replacing $n+1$ with $n + t$ and composing, we see that 
        \[
            \myB^0(X, \Gamma, \cO_X(K_X + nA + L)) \otimes H^0(X, \cO_X(tA)) \to \myB^0(X, \Gamma, \cO_X(K_X + (n+t)A + L))
        \]
        also surjects.  Now $\myB^0(X, \Gamma, \cO_X(K_X + (n+t)A + L))$ globally generates $\tau_+(\omega_X, \Gamma) \otimes \cO_X( (n+t)A + L)$ for $t\gg 0$ by the second part of \autoref{prop.Tau+IndependentOfAmple}.  This completes the proof.  
    \end{proof}

    We use this to deduce the following statement which more closely matches the existing literature in characteristic $p > 0$.

    \begin{corollary}
        \label{cor.GlobalGenerationViaTauOX}
        With notation as above, suppose that $\Delta \geq 0$ and that $M$ is a Cartier divisor so that $M - K_X - \Delta$ is big and semi-ample.  Additionally suppose that $A$ is globally generated and ample.  Then if $n_0$ is the dimension of the fiber $X_{\fram}$, we have for any $n \geq n_0$ that
        \[
            \tau_+(\cO_X, \Delta) \otimes \cO_X(nA + M)
        \]
        is globally generated by $\myB^0(X, \Delta, \cO_X( nA+M))$.
    \end{corollary}
    \begin{proof}
        Choose a Cartier divisor $D \geq 0$ such that $K_X + D \geq 0$.  Then we have that 
        \[
            \tau_+(\omega_X, \Delta + K_X + D)  = \tau_+(\cO_X, \Delta + D) = \tau_+(\cO_X, \Delta) \otimes \cO_X(-D).
        \]
        Set $L = M + D$, and we observe that $L - \Delta - K_X - D = M + D - \Delta -K_X - D = M - K_X - \Delta$ is big and semi-ample.  In particular, we have that 
        \[
            \tau_+(\omega_X, \Delta + K_X + D) \otimes \cO_X(nA + L) = \tau_+(\cO_X, \Delta) \otimes \cO_X(nA + M)
        \] 
        is globally generated by 
        \[
            \myB^0(X, \Delta + K_X + D, \omega_X \otimes \cO_X(nA + L)) = \myB^0(X, \Delta, \cO_X(nA + M))
        \]
        which is what we wanted to show.
    \end{proof}

\subsection{Skoda type result}
In the previous sections we considered $\bigplus$-test ideals of divisor pairs $(X, \Delta)$.  In characteristic zero and $p > 0$, it is also natural to consider pairs $(X, \fra^t)$ where $\fra \subseteq \cO_X$ is an ideal sheaf and $t > 0$ is a real or rational number.  One sometimes even considers triples $(X, \Delta, \fra^t)$.  While we believe the work we have done in previous sections generalizes to this setting of triples, we do not pursue this.  Instead, we prove a Skoda-type result in the case when $\Delta = 0$ for the object we define as $\tau_+(\omega_X, \fra^t)$.

Let $X$ be a normal integral scheme projective over a complete Noetherian local ring $(R, \fram)$.  Suppose that $\frak a\subseteq \cO _X$ is an ideal. 
Let $\mu :X'\to X$ be the normalization of the blow up of $\frak a$ so that $\frak a \cdot \cO_{X'} =\cO _{X'}(-F)$ where $F\geq 0$ is a divisor.   
In particular $\bar {\frak a}=\mu _* \cO _{X'}(-F)$ is the integral closure of $\frak a$ (see \cite[9.6.a]{LazarsfeldPositivity2}).  For any rational number
$t>0$, we define the \emph{$\bigplus$-test submodule associated to $\fra^t$}, denoted $\tau _+(\omega _X, \frak a^t)$, to be the subsheaf of $\omega_X$ such that:
\[
    \tau _+(\omega _X, \frak a^t)\otimes \sL \subseteq \omega_X \otimes \sL
\]
is globally generated by 
\[
    \begin{array}{rcl}
        \myB ^0(X',tF, \mathcal O _{X'}(K_{X'}+ L')) & = & \myB ^0(X',\{ t  F\}, \mathcal O _{X'}(K_{X'}- \lfloor t F \rfloor + L'))\\
        & \subseteq & H^0(X', \cO_{X'}(K_{X'}+ L')) \\
        &\subseteq & H^0(X, \cO_X(K_X +  L))
    \end{array}
\]
where $\sL=\cO_{X}(L)$ is sufficiently ample, $\sL'=\mu ^*\sL$, and $L'=\mu ^*L$.  
Arguing as in previous sections, it is straightforward to see that $\tau _+(\omega _X, \frak a^t)$ is independent of the choice of the sufficiently ample line bundle $\sL$.

Next, suppose that $\sA = \cO_X(A)$ is a sufficiently ample line bundle on $X$ so that $\sA$ and $\fra \otimes \sA$ are globally generated.  Then if $\sA' = \pi^* \sA$ we have that $\sA'(-F) = \cO_{X'}(A' - F)$ is globally generated.  
In some cases, we replace $\sL$ by $\sL \otimes \sA^{\lfloor t \rfloor}$.  This is harmless since we have simply made our line bundle even more ample. 

Instead of setting $X'$ to be the normalized blowup, by \autoref{prop.B0StableBirational}, we can set $X' \to X$ to be any projective birational map from a normal scheme $X'$ dominating the normalized blowup since $\myB^0$ is unaffected.
Note that $\myB ^0(X',tF, \mathcal O _{X'}(K_{X'} + L'))\subseteq H^0(X', \cO_ {X'}(\lceil K_{X'} - tF +L'\rceil))$ and so, we have that 
\[
    \tau _+(\omega _X, \frak a^t) \otimes \sL\subseteq \mu_* \cO_{X'}(\lceil K_{X'} - tF + L'\rceil)
\] 
since the right side is globally generated by a subspace of $H^0(X, \cO _X(K_X+L))$ for $L$ sufficiently ample.  Hence, if for instance $X'$ is a log resolution of $(X, \fra)$, we have an inclusion to the corresponding multiplier submodule $\tau _+(\omega _X, \frak a^t)\subseteq \mathcal J(\omega_X,\frak a^t)$.  


If $X$ is $\bQ$-Gorenstein, one can likewise define the \emph{$\bigplus$-test ideal associated to $\fra^t$}, denoted $\tau_+(\cO_X, \fra^t)$, to be the ideal sheaf $\sJ$ so that $\sJ \otimes \sL$ is globally generated by 
\[
    \myB^0(X', tF, \cO_{X'}(K_{X'} - \mu^* K_X + L')) \subseteq H^0(X, \cO_X(L))
\]
for $\sL = \cO_X(L)$ sufficiently ample.  
If $K_X$ is Cartier, we see directly that $\tau_+(\cO_X, \fra^t) \otimes \omega_X = \tau_+(\omega_X, \fra^t)$.
Arguing as above, we again obtain results of the form:
\[
    \tau _+(\cO_X, \frak a^t)\subseteq \mathcal J(\cO_X,\frak a^t).
\]

\begin{remark}[Even more general definitions]
    In fact, given a $\bQ$-divisor $\Delta \geq 0$ so that $K_X + \Delta$ is $\bQ$-Cartier, and given ideals $\fra_1, \dots, \fra_m$ and rational numbers $t_1, \dots, t_m \geq 0$, we can define the \emph{$\bigplus$-test ideal associated to $(X, \Delta, \fra_1^{t_1} \cdots \fra_m^{t_m})$}, denoted
    \[
        \tau_+(\cO_X, \Delta, \fra_1^{t_1} \cdots \fra_m^{t_m}),
    \]
    to be the ideal sheaf $\sJ$ so that $\sJ \otimes \sL$ is generated by 
    \[
        \myB^0(X', t_1 F_1 + \cdots + t_m F_m, \cO_{X'}(K_{X'} - \mu^* (K_X + \Delta) + L')) \subseteq H^0(X, \sL)
    \]
    for $\sL$ sufficiently ample where $\mu : X' \to X$ is a projective birational map from a normal integral scheme $X$ that factors through the blowups of the ideal sheaves $\fra_i$ and where $\fra_i \cdot \cO_{X'} = \cO_{X'}(-F_i)$.  We won't work in this generality, although we believe that the many results of this paper generalize to this setting.
\end{remark}    

\begin{remark}
    This is not the only natural way to define $\tau_+(\omega_X, \fra^t)$.  For instance, in \cite{MaSchwedePerfectoidTestideal} and \cite{SatoTakagiArithmeticDeformationOFFPure}, when $X = \Spec R$, $\tau_+(\fra^t)$ was essentially defined as the sum of $\tau_+({t \over m}\Div(f))$ where $f \in \fra^{m}$.  This is also roughly equivalent to how the generalized test ideals $\tau(\fra^t)$ of \cite{HaraYoshidaGeneralizationOfTightClosure} are defined in characteristic $p > 0$.  In fact, we use a variant of this approach below in \autoref{sec.+TestModulesOfLinearSeries}.  Our approach in this section instead is inspired by a characterization of test ideals from \cite{SchwedeTuckerTestIdealsNonPrincipal} which itself is inspired by the classical definition of multiplier ideals.  We briefly compare these two constructions in our global setting.  

    Again set $\sA$ so that $\fra \otimes \sA$ is globally generated.  Since $\mu ^*{1 \over m}\Gamma \geq F$ for any $\Gamma \in |\sA^m\otimes \frak a^m| \subseteq |\sA^m|$, by \autoref{lem.SimplePropertiesOfB0} and \autoref{prop.B0StableBirational} we see that $\myB ^0(X',tF, \mathcal O _{X'}(K_{X'}+L'))$ contains
    \[ 
        \myB ^0(X',t\cdot \mu ^*{1 \over m}\Gamma, \mathcal O _{X'}(K_{X'}+L'))\cong \myB ^0(X,{t \over m}\Gamma, \mathcal O _{X}(K_{X} + L)).
    \]
    Putting all this together, one sees that we have a containment 
    \[
        \sum _{\Gamma \in |\sA^m\otimes \frak a^m|}\tau _+(\omega _X,{t \over m}\cdot \Gamma ) \subseteq \tau _+(\omega _X, \frak a^t).
    \]
    We do not know if we always have equality in the inclusion above even in the local case when $X = \Spec R$.  It would also be natural to expect that we have equality without the sum for a single sufficiently general $\Gamma$.
  
    One advantage of our new definition is that it is clear that $\tau_+(\omega_X, \fra^t) = \tau_+(\omega_X, \overline{\fra}^t) = \tau_+(\omega_X, \overline{\fra^t})$ where $\overline{\bullet}$ denotes the integral closure.  This is not clear for the sum-style definition.
\end{remark}

We also have the following.
\begin{lemma} 
    \label{lem.EasyContainmentOfIdealInTestIdeal}
    With notation as above, for any rational $t>0$, we have a natural inclusion
    \[
        \mathfrak a ^{\lceil t \rceil} \cdot \tau_+(\omega_X) \subseteq \tau _+(\omega_X,\mathfrak a ^t).
    \]
    Furthermore, if $X$ is $\bQ$-Gorenstein 
    then 
    \[  
        \mathfrak a ^{\lceil t \rceil} \cdot \tau_+(\cO_X) \subseteq \tau _+(\mathcal O _X,\mathfrak a ^t).
    \]
\end{lemma}
\begin{proof} 
    Since $\tau_+(\omega_X, \fra^{\lceil t \rceil}) \subseteq \tau_+(\omega_X, \fra^t)$ we may replace $t$ by $\lceil t \rceil$ and so assume that $t$ is an integer.
    Fix $\sA$ ample so that $\mathfrak a 
\otimes \sA$ is globally generated, which implies that  $\mathfrak a^{t} 
\otimes \sA^{t }$ is also globally generated, and hence  there is a surjection 
\[
    H^0(X,\mathfrak a ^t
\otimes \sA^t)\otimes \sA^{-t}\to \mathfrak a^t\subseteq \mathcal O _X.
\] 
Thus $\mathfrak a^t=\sum _{\Gamma \in |\sA^t\otimes \frak a^t|} \mathcal O _X( -\Gamma )$.
Arguing as in the remark above, we have that 
\[
    \begin{array}{rcl}
        \mathfrak a^t \cdot \tau_+(\omega_X) & = &  \tau_+(\omega_X) \cdot \left( \sum _{\Gamma \in |\sA^t\otimes \frak a^t|} \mathcal O _X(- \Gamma )\right)  \\
        & = & \sum _{\Gamma \in |\sA^t\otimes \frak a^t|}\tau _+(\omega _X) \cdot \mathcal O _X(- \Gamma ) \\
        & = & \sum _{\Gamma \in |\sA^t\otimes \frak a^t|}\tau _+(\omega _X, \Gamma ) \\
        & \subseteq & \tau _+(\omega _X, \frak a^t).
    \end{array}
\]
 %
%
The statement about $\cO_X$ follows similarly.
\end{proof}
\begin{theorem} 
    \label{thm.SkodaWithoutTooMuchGenerality}
    Suppose that $X$ is normal, integral and projective over a complete local Noetherian domain $(R, \fram)$ and suppose that $0 \neq \fra$ is an ideal sheaf on $X$.  Let $\mu : X' \to X$ denote the normalized blowup of $\fra$ and let $n = \dim X'_{\fram}$.  
    Then for any rational number $t \geq n + 1$ we have $\tau _+(\omega _X, \frak a^t)=\frak a\cdot \tau _+(\omega _X, \frak a^{t-1})$.  

    Furthermore, if $X$ is $\bQ$-Gorenstein, we have that $\tau_+(\cO_X, \fra^t) = \fra \cdot \tau_+(\cO_X, \fra^{t-1})$.
\end{theorem}
In the case that $X = \Spec R$, then $\dim X'_{\fram} \leq \dim X - 1$.  Hence we recover the usual Skoda statement in the local case.  As typical, by iterating this result, if $t$ is an integer, then we have that 
\[
    \fra^{t - n} \cdot \tau_+(\omega_X, \fra^n) = \tau_+(\omega_X, \fra^t)
\]
and in the $\bQ$-Gorenstein case that 
\[
    \fra^{t - n} \cdot \tau_+(\cO_X, \fra^n) = \tau_+(\cO_X, \fra^t).
\]

\begin{proof} 
    We first handle the case of $\tau_+(\omega_X, \dots)$ and then explain how to modify the proof for the $\bigplus$-test ideal case.
    Write $\frak a\cdot \cO _{X'}=\cO _{X'}(-F)$ for some effective divisor $F$. 
    We will assume that $\sA$ and $\sA \otimes \frak a$ are globally generated and $\sL$ is a sufficiently ample line bundle. The proof now mimics the proof of \autoref{t-eff} but we use $X'^+$ instead of $X^+$. 
    
    Note that $\mu ^* H^0(X, \sA \otimes \frak a)\cong H^0(X' ,  \sA' (-F))$ generates $ \sA'(-F)$ and we may pick sections $s_0,\ldots , s_n$ generating  $\sA' (-F)$ (this is because $\dim X'_{\frak m}=n$). Notice if the residue field is not infinite we may need to take an \'etale base change of $R$ as in the proof of \autoref{t-eff}, but this is not an obstruction due to \autoref{prop.FiniteEtalExtensionOfRAndB^0}.
    We then have an exact Skoda/Koszul complex on $X'$ \[
            0 \to \sF_{n+1} \to \sF_{n} \to \dots \to \sF_1 \to \sF_0 \to 0
        \]
        where $\sF_i = \wedge^iV\otimes \sO _{X'}(-i(A'-F))$, $V$ is the free $R = H^0(X, \cO_X)$-module on the $s_i$ so that $\wedge^i V = R^{\oplus {n+1 \choose i}}$, and $A'=\mu ^* A$.  Set $L$ to be a sufficiently ample line bundle on $X$ and $L' = \mu^* L$.  Let $\rho':X'^+\to X'$ be the canonical map and pull back this complex to $X'^+$.  It stays exact since we are pulling back a resolution of a locally free sheaf $\sF_0$. 
        We now twist by $\rho'^*\big((n+1)A'+L' - tF\big) = \rho'^* \big((n+1)(A' - F) + L' - (t - n - 1)F \big)$ which is a line bundle on $X'$, and not just a $\bQ$-line bundle.  Furthermore, there exists a finite morphism $g' : Y' \to X'$ where  $g'^*(L' - (t-i)F)$ is in fact a big and semi-ample Cartier divisor for $i = 0, \dots, n+1$.
        Next apply the functor $\sHom_{\cO_{X'^+}}(\bullet,  \cO_{X'^+})$ to obtain an exact complex  
        \[
            0 \leftarrow \sG_{n+1} \leftarrow \sG_{n} \leftarrow \dots \leftarrow \sG_1 \leftarrow \sG_0 \leftarrow 0
        \]
        where 
        \[
            \begin{array}{rl}
                \sG_i = & \cO_{X'^+}(\rho'^*(-(n+1-i)(A'-F) - L' + (t-n-1)F ))\otimes \wedge ^iV \\
                = & \cO_{X'^+}(\rho'^*(-(n+1-i)A' - L' + (t - i)F ))\otimes \wedge ^iV.
            \end{array}
        \]           
        
        Taking local cohomology we obtain the following commutative diagram
        \[
            {\small
                \xymatrix@C=16pt{
                    V\otimes  H^{d}_{\fram} \big(\myR \Gamma(X'^+, \cO_{X'^+}(\rho'^*(-nA'- L' + (t-1)F )))\big) 
                    & \ar@{_{(}->}[l]_-{\alpha} H^{d}_{\fram} \big(\myR \Gamma(X'^+, \cO_{X'^+}(\rho'^*(-(n+1)A' - L' + tF)))\big) \\
                V\otimes  H^{d}_{\fram} \big(\myR \Gamma(X', \cO_{X'}(-nA' - L'))\big)  \ar[u]_{j_1} 
                    & \ar[l] H^{d}_{\fram}\big(\myR \Gamma(X', \cO_{X'}(-(n+1)A'-L' ))\big) \ar[u]_{j_0}
                .}
            }
        \]
        Here $\alpha $ is injective by Bhatt vanishing (\autoref{BhattVanishing}) since $A'$ and $g'^*(L' - (t-i)F)$ are big and semi-ample (note $L$ was arbitrarily ample).

        Applying Matlis duality $\bullet^{\vee} = \Hom(\bullet, E)$ where $E = E_{R/\fram}$ is the injective hull of the residue field of $R$, we see that
        \[
           V\otimes  \myB^0(X', (t-1)F, \cO_{X'}(K_{X'} + nA' + L')) = \Image(j_1)^{\vee} \twoheadrightarrow \] 
           \[ 
               \Image(j_0)^{\vee} = \myB^0(X', tF, \cO_{X'}(K_{X'} + (n+1)A' + L'))
        \]
        surjects.  Pushing forward via $\mu : X' \to X$ (and remembering that the sections $s_i$ are elements of $H^0(X, \sA\otimes \mathfrak a) \subseteq H^0(X, \sA)$), we have a surjection 
        \[
            \tau _+(\omega _X, \frak a^{t- 1})\otimes \sA^{n}\otimes \sL \otimes H^0(X,  \sA\otimes \mathfrak  a) \to \tau _+(\omega _X, \frak a^{t})\otimes \sA^{n+1}\otimes \sL.
        \]  
        It then follows that $\frak a \cdot \tau _+(\omega _X, \frak a^{t-1})=\tau _+(\omega _X, \frak a^{t})$.
         This completes the proof of the first case.  
         
         The second case follows similarly but replace the sufficiently ample $L$ by $L - K_X$ (and hence $L'$ by $L' - \mu^* K_X$).  In particular, our $Y'$ should also must such that the pullback of $\mu^*K_X$ is Cartier.  The bottom row of the analogous commutative diagram also must round down the divisors since $\mu^* K_X$ is only a $\bQ$-divisor.  
        The rest of the proof then follows analogously.
\end{proof}


\subsection{An application to the Brian{\c{c}}on-Skoda theorem}

    As suggested by S.~Takagi, we also show how the above can be used to show the following version of the Brian{\c{c}}on-Skoda theorem for $\bQ$-Gorenstein splinters of mixed characteristic.  Since splinters are pseudo-rational by \cite{BhattAbsoluteIntegralClosure}, the Brian{\c{c}}on-Skoda theorem for splinters follows from the work of Lipman-Tessier \cite[Theorem 2.1, Corollary 2.2]{LipmanTeissierPseudorationallocalringsandatheoremofBrianconSkoda}.  However, instead of simply looking at the analytic spread of the ideal, the Lipman-Tessier bound also considers the codimension of the associated primes of \emph{powers} of $I$, which do not appear in our corollary.  Related versions of this result also follows from \cite[Theorem 2.7]{HeitmannMaExtendedPlusClosure}, \cf \cite{HeitmannPlusClosureMixedChar} and \cite[Section 4]{MaSchwedeTuckerWaldronWitaszekAdjoint}.

    \begin{corollary}
        Suppose that $(R, \fram)$ is an complete normal $\bQ$-Gorenstein local domain of residual characteristic $p > 0$ and $X = \Spec R$.  Suppose that $I \subseteq R$ is an ideal of analytic spread\footnote{The analytic spread $m$ of $I$ is the dimension $R[It]/\fram R[It]$ where $R[It]$ is the Rees algebra of $I$.  Hence $m - 1 = n = \dim X'_{\fram}$.} $n+1$, then for any integer $j > 0$
        \[
            \overline{I^{j+n}} \cdot \tau_+(\cO_X) \subseteq I^{j}
        \]
        where $\overline{\bullet}$ denotes integral closure.  Furthermore, if $(R, \fram)$ not necessarily complete, but is instead an excellent local splinter (meaning it is a summand of every module finite ring extension, or equivalently if $X$ is $\bigplus$-regular / $T$-regular), then we have that 
        \[
            \overline{I^{j+n}}  \subseteq I^{j}.
        \]
    \end{corollary}    
        The proof we use is essentially the same as the one contained in \cite[Section 9.6]{LazarsfeldPositivity2}.
    \begin{proof}
        For the second statement, we may also assume that $R$ is complete (the splinter property is preserved by \cite{DattaTuckerPermanence}).  We will now use the theory we have sofar developed, applied in the case that $X = \Spec R$.  
        We have that $\overline{I^m} \cdot \tau_+(\cO_X) \subseteq \tau_+(\cO_X, \overline{I^m}) = \tau_+(\cO_X, I^m)$ for all $m$ by \autoref{lem.EasyContainmentOfIdealInTestIdeal}.  Therefore 
        \[
            \overline{I^{j+n}} \cdot \tau_+(\cO_X)  \subseteq \tau_+(\cO_X, {I^{j+n}}) = I^{j} \cdot \tau_+(\cO_X, I^{n}) \subseteq I^{j}.
        \]                
        When $R$ is a splinter, we have that $\tau_+(\cO_X) = \cO_X$ (since every finite ring extension of $R$ splits) and so the proof is complete.
    \end{proof}

    If $X$ is not $\bQ$-Gorenstein but $\Delta \geq 0$ is such that $K_X + \Delta$ is $\bQ$-Cartier, we believe a slight generalization of the same proof also shows that $\overline{I^{j+n}} \cdot \tau_+(\cO_X, \Delta) \subseteq I^j$.

\section{$\bigplus$-test modules of linear series}
\label{sec.+TestModulesOfLinearSeries}

   In this section, we will generalize results of \cite{LazarsfeldPositivity2,MustataNonNefLocusPositiveChar,SatoStabilityOfTestIdeals} on the global generation of test ideals (or multiplier ideals) of linear series, from the equal characteristic case to the mixed characteristic case.  We then apply this to study the order of vanishing of linear series and to the diminished base locus as in \cite{NakayamaZariskiDecompAndAbundance,EinLazMusNakPopAsymptoticInvariantsOfBaseLoci,MustataNonNefLocusPositiveChar}, note some related results were also obtained in the singular case in \cite{CacciolaDiBiagioAsymptoticBaseLociOnSingular,SatoStabilityOfTestIdeals,MurayamaTheGammaConstruction}.  Many of our arguments were inspired by those in \cite{MustataMultiplierIdealOfASum}.  We begin by defining the corresponding test ideals for linear series.  

    \begin{setting}
        \label{set.TestModulesOfLinearSeries}
        Throughout this section, suppose $X$ is a Noetherian normal integral scheme projective over a complete local Noetherian ring $(R, \fram)$ of mixed characteristic $(0, p)$.  
        We \emph{also} assume that $X$ is a regular scheme (in other words, it is nonsingular) and $\sL=\mathcal O _X(L)$ is a line bundle.
    \end{setting}

    \begin{definition} 
        Let $W\subseteq H^0(X,\sL)$ be a subspace and $\lambda >0$ a positive rational number, then the $\bigplus$-test module corresponding to the linear series $|W|$ is
        \[\tau _+(\omega _X , \lambda \cdot |W|) := \sum _{D\in |W|}\tau _+(\omega _X ,  \lambda \cdot D).\]
        Since $X$ is Noetherian, we may pick finitely many elements $D_1,\ldots ,D_r\in |W|$ such that
        $\tau _+(\omega _X ,  \lambda \cdot |W|) = \sum _{i=1}^r\tau _+(\omega _X ,  \lambda \cdot D_i)$.

        We also define $\tau_+(\cO_X, \lambda \cdot |W|) = \tau_+(\omega_X, \lambda \cdot |W|) \otimes \omega_X^{-1}$.  
    \end{definition}

    \begin{remark}
        Note, we could have also defined $\tau_+(\omega_X, \lambda \cdot |W|)$ as in the previous section as $\tau_+(\omega_X, \frb_{|W|}^\lambda)$.  This is what is done in the classical situation of the multiplier ideal, and also what was done in positive characteristic in \cite{MustataNonNefLocusPositiveChar,SatoStabilityOfTestIdeals} (although in those cases, the sum formulation of the multiplier ideal as above is known to agree with $\tau_+(\omega_X, \frb_{|W|}^\lambda)$).  While we expect similar statements for that latter definition, we do not pursue that approach here.
    \end{remark}

If $D$ is a $\mathbb Q$-divisor with $\kappa (D)\geq 0$, then $mD$ is a divisor such that $|mD|\ne \emptyset$ for all $m>0$ sufficiently divisible.
We let 
\[
    \tau_+  (\omega _X,\lambda \cdot || D||)=\bigcup _{m>0}\tau_+   (\omega _X, \frac \lambda m \cdot | mD|)
\] 
where we ignore the terms such that $|mD|=\emptyset$ (or equivalently we let $\tau_+   (\omega _X, \frac \lambda m \cdot | mD|)=0$ if $|mD|=\emptyset$).  Recall that if $mD$ is not necessarily an integral divisor, then $|mD|=|\lfloor mD\rfloor|+\{mD\}$. We also let 
\[
    \tau_+  (\cO_X,\lambda \cdot || D||)=\bigcup _{m>0}\tau_+   (\cO_X, \frac \lambda m \cdot | mD|) = \tau_+(\omega_X, \lambda \cdot || D||) \otimes \omega_X^{-1}.
\] 

\begin{lemma}\label{l-asy} We have $\tau_+  (\omega _X,\lambda \cdot || D||)=\tau_+   (\omega _X, \frac \lambda m \cdot | mD|)$ for any $m>0$ sufficiently divisible.\end{lemma}
\begin{proof} Note that if $G\in |mD|\ne \emptyset$, then for any $r>0$, we have $rG\in |rmD|$ and therefore
\[ \tau_+   (\omega _X, \frac \lambda m \cdot | mD|)=\sum _{G\in |mD|}\tau_+   (\omega _X, \frac \lambda m  G)=\sum _{G\in |mD|} \tau_+   (\omega _X, \frac \lambda {mr} \cdot rG)\subseteq 
 \tau_+   (\omega _X, \frac \lambda {mr} \cdot | mrD|).\]
 Because $X$ is Noetherian, the set $\{ \tau_+   (\omega _X, \frac \lambda m \cdot | mD|)\ s.t.\ |mD|\ne \emptyset\}$ has a maximal element and the claim follows. \end{proof}
 
 \begin{lemma} \label{l-easy}
    With the above notation (and continuing to assume that $X$ is nonsingular)
    \begin{enumerate}
        \item If $|D|\ne \emptyset$ then  $\tau_+  (\omega _X,\lambda \cdot | D|)\subseteq \tau_+  (\omega _X,\lambda \cdot || D||)$.\label{l-easy.1}
        \item  If $D$ has $\bZ$-coefficients and $|D|$ is non-empty, then
            $\tau_+  (\omega _X, | D|)= \frak b_{|D|}\otimes \omega _X$ or in other words $\tau_+(\cO_X, |D|) = \frak{b}_{|D|}$.\label{l-easy.2}
        \item If $\lambda <\mu$, then 
            $\tau_+  (\omega _X,\mu \cdot || D||)\subseteq \tau_+  (\omega _X,\lambda \cdot || D||)$.\label{l-easy.3}
        \item Let $k>0$ be an integer, then   $\tau_+  (\omega _X,\lambda \cdot || D||)=\tau_+  (\omega _X,\frac \lambda k \cdot ||k D||)$.\label{l-easy.4}
        \item If $L-\lambda D$ is big and semiample and $A$ is ample and globally generated, then 
            \[
                \tau_+(\omega _X,\lambda \cdot || D||)\otimes \cO _X(nA+L) = \tau_+(\cO_X,\lambda \cdot || D||)\otimes \cO _X(L + K_X + nA) 
            \]         
            is generated by a sub linear series of $H^0(X,\cO _X(K_X + nA+L))$ where $n=\dim X_{\frak m}$.\label{l-easy.5}
        \item If $D-G$ is semiample, then  $\tau_+  (\omega _X,\lambda \cdot || G||) \subseteq  \tau_+  (\omega _X,\lambda \cdot || D||)$.\label{l-easy.6}
    \end{enumerate}
\end{lemma}
\begin{proof} 
    Claim \autoref{l-easy.1} follows from the definition.

    Claim \autoref{l-easy.2} follows because $X$ is nonsingular and so by \autoref{lem.SimplePropertiesOfB0_D},
    \[\tau_+  (\omega _X, | D|)=\sum_{G\in |D|} \tau_+  (\omega _X,  G)=\sum  \omega _X (-G) =
    \left( \sum  \mathcal O _X (-G)\right) \otimes \omega _X= \frak b_{|D|}\otimes \omega _X.\]

    Claim \autoref{l-easy.3} follows because for any $m>0$ sufficiently divisible, we have
    \[\tau_+  (\omega _X,\lambda \cdot || D||)=\tau_+  (\omega _X,\frac \lambda {m} \cdot |m D|) \supset \tau_+  (\omega _X,\frac \mu {m} \cdot |m D|)=\tau_+  (\omega _X,\mu \cdot || D||).\]

    Claim \autoref{l-easy.4} follows because  for any $m>0$ sufficiently divisible, we have \[\tau_+  (\omega _X,\lambda \cdot || D||)=\tau_+  (\omega _X,\frac \lambda {mk} \cdot |mk D|)=\tau_+  (\omega _X,\frac {\lambda/k} {m} \cdot |m(k D)|)=\tau_+  (\omega _X,\frac \lambda k \cdot ||k D||).\]

    To see \autoref{l-easy.5}, for some $G\in |mL|$ note that by \autoref{t-eff},
    $\tau _+(\omega _X,\frac \lambda m \cdot G)\otimes \cO _X(nA+L)$ is generated by $\mathbf B^0(X,\frac \lambda m \cdot G , \omega _X \otimes \cO _X(nA+L))$. Thus $\tau_+  (\omega _X,\lambda \cdot || D||)\otimes \cO _X(nA+L)$ is generated by \[\sum _{i=1}^r {\mathbf B}^0(X,\frac \lambda m \cdot G _i, \omega _X \otimes \cO _X(nA+L))\subseteq H^0(X,\cO_X (K_X+nA+L))\] for appropriately chosen $G_i\in |mL|$.
    
    We will now prove \autoref{l-easy.6}. For $m>0$ sufficiently divisible, we have
    \[ \tau_+  (\omega _X,\lambda \cdot || G||) =\tau_+  (\omega _X,\frac \lambda m \cdot |m G|), \qquad  \tau_+  (\omega _X,\frac \lambda m \cdot | mD|) = \tau_+  (\omega _X,\lambda \cdot || D||).\]
    Since $D-G$ is semiample, we may assume that $|m(D-G)|$ is base point free. For any $F\in  |m(D-G)|$ we have an inclusion $W:=|mG|+F\subseteq |mD|$.
    Let $c=\lceil \frac \lambda m\rceil$ and pick divisors $D_1,\ldots, D_r\in |mD|$ and $G_1,\ldots , G_s\in |mG|$ such that \[\tau _+(\omega _X ,\lambda \cdot ||D||)=\sum _{i=1}^r \tau _+(\omega _X ,\frac \lambda m \cdot D_i), \qquad \tau _+(\omega _X ,\lambda \cdot ||G||)=\sum _{i=1}^s \tau _+(\omega _X ,\frac \lambda m \cdot G_i).\]
    Possibly adding some $G_i+F$ to the set $\{D_1,\ldots , D_r\}$, we may also assume that $\{G_1+F,\ldots , G_s+F\}\subseteq \{ D_1,\ldots , D_r\}$.
    Thus
      \[
          \tau _+(\omega _X ,\lambda \cdot ||G||)\otimes \mathcal O _X(-cF)=
       \sum _{i=1}^s \tau _+(\omega _X ,\frac \lambda m \cdot G_i)\otimes \mathcal O _X(-cF)=
       \sum _{i=1}^s \tau _+(\omega _X ,\frac \lambda m \cdot G_i+c\cdot F) 
       \]
       \[
        \subseteq \sum _{i=1}^s \tau _+(\omega _X ,\frac \lambda m \cdot (G_i+ F))\subseteq \sum _{i=1}^r \tau _+(\omega _X ,\frac \lambda m \cdot D_i)= \tau _+(\omega _X ,\lambda \cdot ||D||).
    \]
   Since $|m(D-G)|$ is base point free, it follows that $\sum _{F\in |m(D-G)|}\mathcal O _X(-F)=\mathcal O _X$ and hence 
    \[
        \tau_+  (\omega _X,\lambda \cdot || G||) =\sum _{F\in |m(D-G)|}\tau_+  (\omega _X, \lambda \cdot  || G||) \otimes \cO _X(-cF) \subseteq \tau_+  (\omega _X, \lambda \cdot || D||).
    \]
    
\end{proof}

\begin{theorem} 
    \label{thm.OrdIndependence}
    Let $Z$ be a closed integral subscheme of a regular scheme $X$ over ${\rm Spec }(R)$ and $D$ a big $\mathbb Q$-divisor. Then ${\rm ord}_Z(||D||)$ is independent of the representative in the numerical class of $D$.
\end{theorem}
 \begin{proof} 
        Recall that \[ {\rm ord}_Z(||D||)=\lim \frac 1 m {\rm ord}_Z(|mD|)={\rm inf }\{ \frac 1 m {\rm ord}_Z(|mD|) \}.\]
    Let $E\equiv D$ be numerically equivalent big $\mathbb Q$-divisors. We must show that 
    ${\rm ord}_Z(||D||)={\rm ord}_Z(||E||)$. Replacing $D,E$ by appropriate multiples, we may assume that $D,E$ have integer coefficients. Pick $l>0$ such that $lD-(K_X+(n+1)A)\sim G>0$ where $A$ is ample and globally generated and $n=\dim X_{\frak m}$.
    Then $mE-G\equiv (m-l)D+K_X+(n+1)A$.
    In particular 
    \[
        m E - G - K_X - (m-l)D - nA \equiv A
    \] 
    is ample (and so big and semi-ample).

    Since, by \autoref{l-easy}(e),  $\tau _+(\cO_X,(m-l)\cdot ||D||)\otimes \cO _X(mE-G)$ is globally generated, it follows that 
    \[
        \tau _+(\cO_X,(m-l)\cdot ||D||)\subseteq \frak b_{|mE-G|}
    \] 
    where $\frak b_{|mE-G|}$ is the base ideal of the linear series $|mE-G|$.
    Thus 
    \[ 
        \cO _X(-G)\cdot \tau _+(\cO_X,(m-l)\cdot ||D||)\subseteq \cO _X(-G)\cdot\frak b_{|mE-G|}\subset
        \frak b_{|mE|},
    \]
    where the last inequality is induced by the multiplication map $|G|\times  |mE-G|\to |mE|$.
    
    It follows by \autoref{l-easy}(a--d), that for $m>l$ we have 
    \[
        \cO _X(-G)\cdot \frak b_{|mD|}\subseteq \cO _X(-G)\cdot \frak \tau_+(\cO_X,||mD||) \subseteq 
    \cO_X(-G)\cdot \frak \tau _+(\cO_X,||(m-l)D||)   \subseteq \frak b_{|mE|}.
    \]
    But then \[{\rm ord}_Z(G)+{\rm ord}_Z(|mD|)\geq {\rm ord}_Z(|mE|)\] and passing to the limit we obtain ${\rm ord}_Z(||D||)\geq {\rm ord}_Z(||E||)$. The reverse inequality follows by symmetry and this concludes the proof.
\end{proof}

\begin{theorem} 
    \label{thm.ZInDiminishedBaseLocus}
    Suppose that $R/\frak m$ is infinite. Let $Z$ be a closed irreducible reduced subscheme with associated generic point $\eta_Z$. 
    If $D$ is big then the following are equivalent
    \begin{enumerate}
        \item $Z\not \subseteq {\mathbf B}_-(D)$,
        \item there exists a divisor $G$ such that $Z\not \subseteq {\rm Bs}|mD+G|$ for all $m\geq 1$,
        \item there exists an integer $M>0$ such that ${\rm ord}_Z(|mD|)\leq M$ for all $m\gg 0$,
        \item  ${\rm ord}_Z(||D||)=0$, and
        \item $\tau _+(\cO _X,\lambda \cdot ||D||)$ does not vanish along $Z$ for any $\lambda >0$.
    \end{enumerate}
 \end{theorem}
 \begin{proof} 
 
 (b) implies (c). Since $D$ is big, we may pick an integer $m_0>0$ and  a divisor $T\in |m_0D-G|$.
 For any $m>0$ we may pick a divisor 
 $S_m\in |mD+G|$ whose support does not contain $Z$, then ${\rm ord}_Z(|(m+m_0)D|)\leq {\rm ord}_Z(T+S_m)=M$ where $M={\rm ord}_Z(T)$.
 
  (c) implies (d). This is immediate since ${\rm ord}_Z(||D||)=\lim ({\rm ord}_Z(|mD|)/m)\leq \lim M/m=0$.
  
   (d) implies (e).
   For inductive purposes, we weaken the hypothesis that $X$ is regular.  Instead, we merely assume that $X$ is regular in a neighborhood of $\eta_Z$.
   We may assume that $\lambda =1$. Since $X$ is regular on a neighborhood of $\eta _Z$ there is an isomorphism $\omega _X\cong \mathcal O _X$ on such a neighborhood. 
   Fix $m>0$ such that ${\rm ord}_Z(|mD|)<m$. Let $G\in |mD|$ be a sufficiently general section, then ${\rm ord}_Z(G/m)<1$.
   We now proceed by induction on the codimension of $Z$ in $X$.
   If $Z$ is a divisor in $X$, then on a neighborhood of $\eta _Z$, the pair $(X,G/m)$ has normal crossings with coefficients $< 1$, and so on this neighborhood we have 
   \[
       \mathcal O _X\subseteq \tau _+(\cO _X, |G/m|)\subseteq \tau _+(\cO _X, ||D||)
    \]
   and the claim follows.
   We will now assume that $\dim X-\dim Z\geq 2$ and we will proceed by induction restricting to hyperplanes. Let $H$ be sufficiently ample and $Y\in |H\otimes \mathcal I _Z|$ a sufficiently general hypersurface vanishing along $Z$. We may assume that $Y$ is regular on a neighborhood of $\eta _Z$, and that ${\rm ord}_Z(G|_Y)={\rm ord}_Z(G)<m$. 
   There is an inclusion
  \begin{equation} \label{e-1} 
    {\rm adj}^Y_+(\cO _X,Y+G/m)\subseteq \tau _+(\cO _X,(1-\epsilon) Y+G/m)\subseteq \tau _+(\cO _X,G/m)
\end{equation}
   and we have the equality
   \begin{equation} \label{e-2} 
    {\rm adj}^Y_+(\cO _X,Y+G/m)\cdot \mathcal O _{Y^{\Norm}}=\tau _+(\cO _{Y^{\Norm}}, {\rm Diff}_{Y^{\Norm}}(Y+G/m)). 
\end{equation}
   Note that, on a neighborhood of $\eta _Z$, we have $Y^{\Norm}\to Y$ is an isomorphism. By induction on the dimension, $\tau _+(\cO _{Y^{\Norm}}, {\rm Diff}_{Y^{\Norm}}(Y+G/m))=\mathcal O _Y$ in a neighborhood of $\eta_Z$ and so by \eqref{e-2} $ {\rm adj}^Y_+(\cO _X,Y+G/m)\cong \mathcal O _X$. From the inclusion \eqref{e-1}, it follows that $ \tau _+(\cO_X,G/m)=\mathcal O _X$. Repeating this argument $\dim X -\dim Z - 1$ times, we may assume that $Z$ is a divisor in $X$.
   

      (e) implies (b). Pick $G=K_X+(d+1)H$ where $H$ is very ample and $d=\dim X_{\frak m}$. Then $\tau _+(\mathcal O_X,||mD||)\otimes \cO _X(mD+G)$ is globally generated and so $Z$ is not contained in the base locus of $|mD+G|$.
 
  (b) implies (a). Fix $A$ an ample $\mathbb Q$-divisor. Pick $m>0$ such that $mA-G$ is ample.
  Then ${\mathbf B}(D+A)={\mathbf B}(m(D+A))\subseteq {\mathbf B}(mD+G)$. Thus $Z\not \subseteq {\mathbf B}(D+A)$ and the claim follows as ${\mathbf B}_-(D)=\bigcup {\mathbf B}(D+A)$.
  
   (a) implies (d). By assumption  $Z\not \subseteq {\mathbf B}(D+\frac 1 i A)$ for any $i>0$ where $A$ is a fixed very ample divisor. Since $D$ is big, there is an integer $j>0$ such that $jD-A\sim G\geq 0$. But then $|(k(i+j))D|\supset kG+|ki(D+\frac 1 i A)|$ so that 
   ${\rm ord}_Z(|(k(i+j))D|)  \leq k\cdot {\rm ord}_Z(G)$ for any $k>0$ sufficiently divisible. Then  
   \[  
        {\rm ord}_Z(||D||)\leq
        \frac{{\rm ord}_Z(|(k(i+j))D|)}{k(i+j)} \leq  \frac{k \cdot {\rm ord}_Z(G)}{k(i+j)}= \frac{ {\rm ord}_Z(G)}{i+j}.
   \]
   Note that as $j$ is fixed, the right hand side tends to 0 as $i$ goes to infinity and hence $ {\rm ord}_Z(||D||)=0$.
 \end{proof}

 In our proof of (d) implies (e) above, we showed the following.  Compare with \cite[Theorem 1.1]{KadyrsizovaKenkelPageSinghSmith-ExtramalSings} and also \cite{SatoStabilityOfTestIdeals}.

 \begin{corollary} 
    \label{cor.MultiplicityBoundOnThreshold}
    With the same assumptions of \autoref{thm.ZInDiminishedBaseLocus}.
 If $G\geq 0$ is a $\Q$-divisor such that ${\rm ord}_Z(G)<1$, then $\tau _+(\mathcal O _X, G)=\mathcal O _X$ on a neighborhood of $Z$.
 \end{corollary}
 \begin{proof} The proof follows immediately along the lines of the proof that (d) implies (e) in \autoref{thm.ZInDiminishedBaseLocus}.
 \end{proof}

 \section{Further open questions}

We mention several of the most important open questions associated with the theory developed in this paper.

\subsection{Localization and completion}

    Suppose $X$ is projective over a complete local domain $(R, \fram)$.  For any point on $x \in X$, we have can define $\tau_+(\cO_{X,x}, B)$ as the stalk of $\tau_+(\cO_{X}, B)$ at $x$.  We state what we hope happens when we then complete at $x$.

    \begin{conjecture}
        \label{conj.CommuteWithCompletion}
        Suppose that $X$ is a normal integral scheme of finite type over a complete local ring $(R, \fram)$ such that $R/\fram$ has positive characteristic.  Further suppose that $B \geq 0$ is a $\bQ$-divisor on $X$ such that $K_X + B$ is $\bQ$-Cartier.  Fix a point $x \in X$ with residual characteristic $p > 0$ and let $S = \cO_{X,x}$ denote the stalk at $x$.  Then we have that
        \[
            {\tau_+(\cO_{X}, B)_x} \cdot \widehat{S} = \tau_+(\widehat{S}, \widehat{B}) = \tau_{\widehat{R^+}}(\widehat{S}, \widehat{B})
        \]
        and that
        \[
            \tau_+(\cO_{X}, B)_x = \tau_+(\widehat{S}, \widehat{B}) \cap S.
        \]
    \end{conjecture}

\subsection{Perturbed test modules}

    \begin{conjecture}
        Suppose $X$ is a normal integral scheme projective over a complete local ring $(R, \fram)$ with residual characteristic $p > 0$.  Then for $B$ a $\bQ$-divisor on $X$:
        \[
            \tau_+(\omega_X, B) = \tau_+(\omega_X, B + \epsilon D)
        \]
        for $D > 0$ a Cartier divisor on $X$ and $1 \gg \epsilon > 0$.  
    \end{conjecture}
    
    Even in the case that $X = \Spec R$, this is not clear.   In that case, instead of using $X^+$, one uses a sufficiently large perfectoid big Cohen-Macaulay $S$-algebra, the conjecture is true, see \cite{MaSchwedeSingularitiesMixedCharBCM}.  One way to handle this in the case that $X \neq \Spec R$ would be to try to compute a variant of $\tau$ from a graded perfectoid big Cohen-Macaulay $S$ algebra where $S = \bigoplus H^0(X, \sL)$ is a section ring of $X$ with respect to an ample line bundle $\sL$.
    
    Alternately, we can bake a perturbation term into the definition as in \cite{MaSchwedePerfectoidTestideal}.  Indeed, suppose we have notation as in the conjecture above.  
    For a fixed $D > 0$, and $\epsilon_1 > \epsilon_2$ we see that $\tau_+(\omega_X, \Gamma+\epsilon_1 D) \subseteq \tau_+(\omega_X, \Gamma + \epsilon_2 D) \subseteq \omega_X$.  Thus at least the object $\tau_+(\omega_X, \Gamma + \epsilon D)$ is well defined for sufficiently small $\epsilon$ by the Noetherian property of $\omega_X$.

\subsection{Subadditivity}
    
    \begin{conjecture}
        Suppose that $X$ is regular.  Then
        \[
            \tau_+(\cO_X, \Delta_1 + \Delta_2) \subseteq \tau_+(\cO_X, \Delta_1) \cdot \tau_+(\cO_X, \Delta_2).
        \]
    \end{conjecture}
        We expect this based on the fact that the same formula holds for multiplier ideals in characteristic zero and test ideals in characteristic $p > 0$.  In the local case, when $X = \Spec R$, it is proven in \cite[Section 5]{MaSchwedePerfectoidTestideal} for a variant of perturbed test ideals, using Andr\'e's $A_{\infty}$ instead of $R^+$, \cite{AndreDirectsummandconjecture,BhattDirectsummandandDerivedvariant}, but the method works for us when $X = \Spec R$.  Indeed, generalizations of this approach to certain general mixed $\bigplus$-test ideals has since been accomplished by Murayama in \cite{MurayamaSymbolicTestIdeal}.   All of these results used a strategy similar to, and inspired by, that of S.~Takagi in characteristic $p > 0$ \cite{TakagiFormulasForMultiplierIdeals} (also see \cite[Proposition 2.11(d)]{BlickleMustataSmithDiscretenessAndRationalityOfFThresholds} for an equivalent-up-to-duality proof in characteristic $p > 0$, or \cite{HaraYoshidaGeneralizationOfTightClosure} for a characteristic $p > 0$ proof corresponding more closely to the classical subadditivity proof for multiplier ideals \cite{DemaillyEinLazSubadditivity}).

 \subsection{Discreteness and rationality of jumping numbers}

    Based on the characteristic $0$ and $p > 0$ contexts, we would expect the following to hold.

    \begin{conjecture}
        Suppose that $X$ is $\bQ$-Gorenstein and $\Delta \geq 0$ is a $\bQ$-Cartier $\bQ$-divisor.  The set of real numbers $t > 0$ such that $\tau_+(\cO_X, t\Delta) \neq \tau_+(\cO_X, (t-\epsilon)\Delta)$ for all $1 \gg \epsilon > 0$, are a set of rational numbers without accumulation points in $\bR$.
    \end{conjecture}
    
    In characteristic zero, the corresponding statement about the multiplier ideal follows immediately from the fact that multiplier ideals can be computed on a single resolution.  In characteristic $p > 0$ however, the corresponding statement requires a more involved proof.  

\subsection{Alternate descriptions}
    Suppose $R$ is a complete local domain with positive characteristic residue field and $S$ is a normal finite type $R$-algebra and $\Delta \geq 0$ a $\bQ$-divisor on $\Spec S$ such that $K_S + \Delta$ is $\bQ$-Cartier.  By projectively compactifying $\Spec S$ over $\Spec R$, we can define a test ideal on $\Spec S$ and hence $\tau_+(S, \Delta) \subseteq S$ an ideal of $S$.  Note it could happen that $S$ is purely characteristic zero.

    \begin{question}
        Is there a description of $\tau_+(S, \Delta)$ that does not rely on a compactification of $\Spec S$ over $\Spec R$?  For instance, does $\tau_+(S, \Delta)$ agree with the $T^0(\Spec S, \Delta, \cO_{\Spec S})$ as defined in \cite{TY-MMP}?  Is there a description of $\tau_+(S, \Delta)$ via some closure operation as in tight closure?  
    \end{question}

    A positive answer to \autoref{conj.CommuteWithCompletion} might also shed light on this question.

\bibliographystyle{skalpha}
\bibliography{MainBib}

\end{document}